\documentclass[11pt,english,a4paper]{smfart}

\usepackage[french,main=english]{babel}
\usepackage[utf8, latin1]{inputenc}

\usepackage{graphicx}
\usepackage{ae,amsfonts,euscript,enumerate}
\usepackage{amssymb}
\usepackage{amsmath}
\usepackage[cm]{aeguill}
\usepackage{xcolor}

\renewcommand{\theequation}{\thesection.\arabic{equation}}

\newcommand\qbi[3]{{{#1}\atopwithdelims[]{#2}}_{#3}}

\newtheorem{thm}{Theorem}[section]

\newtheorem*{thmC}{Theorem C}

\newtheorem{defn}[thm]{Definition}

\newtheorem{prop}[thm]{Proposition}
\newtheorem{propdef}[thm]{Proposition-definition}
\newtheorem{lem}[thm]{Lemma}
\newtheorem{cor}[thm]{Corollary}

\newtheorem{rem}[thm]{Remark}

\numberwithin{equation}{section}
\usepackage{fullpage}
\usepackage[T1]{fontenc}
\usepackage{graphicx}
\usepackage{amsmath}
\usepackage{amsfonts}
\usepackage{amssymb}
\usepackage{a4wide}
\DeclareMathOperator{\lcm}{lcm}

\author{B. Adamczewski}
\address{
Univ Lyon, Universit\'e Claude Bernard Lyon 1, CNRS UMR 5208, Institut Camille Jordan, F-69622 Villeurbanne Cedex, France}
\email{Boris.Adamczewski@math.cnrs.fr}

\author{J. P.~Bell}
\address{
Department of Pure Mathematics, University of Waterloo, Waterloo, ON, Canada, N2L 3G1}
\email{jpbell@uwaterloo.ca}

\author{\'E. Delaygue}
\address{
Univ Lyon, Universit\'e Claude Bernard Lyon 1, CNRS UMR 5208, Institut Camille Jordan, F-69622 Villeurbanne Cedex, France}
\email{delaygue@math.univ-lyon1.fr}

\author{F. Jouhet}
\address{
Univ Lyon, Universit\'e Claude Bernard Lyon 1, CNRS UMR 5208, Institut Camille Jordan, F-69622 Villeurbanne Cedex, France}
\email{jouhet@math.univ-lyon1.fr}

\title{Cyclotomic valuation of $q$-Pochhammer symbols and $q$-Integrality of basic hypergeometric series}
\date{}

\thanks{This project has received funding from the ANR project De Rerum Natura (ANR-19-CE40-0018).} 



\begin{document}

\begin{abstract}
We give a formula for the cyclotomic valuation of $q$-Pochhammer symbols in terms of (generalized) Dwork maps. 
We also obtain a criterion for the $q$-integrality of basic hypergeometric series in terms of certain step functions, which generalize Christol step functions. 
This provides suitable $q$-analogs of two results proved by Christol: a formula for the $p$-adic valuation of Pochhammer symbols and 
a criterion for the $N$-integrality of hypergeometric series. 
\end{abstract}

\maketitle

\section{Introduction}\label{sec:intro}

\emph{Factorial ratios} form a remarkable class of sequences appearing regularly in combinatorics, number theory (\textit{e.g.} \cite{ASI, Bober, Chebyshev, Soundararajan}), mathematical physics and geometry (\textit{e.g.} \cite{BvS,Candelas, Delaygue1}). They are sequences of rational numbers of the form 
$$
Q_{e,f}(n):=\frac{(e_1n)!\cdots(e_vn)!}{(f_1n)!\cdots(f_wn)!}  \quad\quad  n\geq 0\,,
$$
where $v$ and $w$ are non-negative integers, and $e:=(e_1,\ldots,e_v)$ and $f:=(f_1,\ldots,f_w)$ are vectors whose coordinates are positive integers. 
Understanding how arithmetic properties of factorial ratios may depend on the integer parameters $e_i$ and $f_i$ leads to 
interesting and challenging problems.  
Landau \cite{Landau} introduced the step function 
\begin{equation}\label{eq:deltadef}
\Delta_{e,f}(x):=\sum_{i=1}^v\lfloor e_ix\rfloor-\sum_{j=1}^w\lfloor f_jx\rfloor  
\end{equation} 
and proved that the $p$-adic valuation of factorial ratios is given by 
\begin{equation*}\label{eq: Galois stable q=1}
v_p\left(Q_{e,f}(n)\right)=\sum_{\ell=1}^\infty\Delta_{e,f}\left(\frac{n}{p^\ell}\right) \,.
\end{equation*}
This result generalizes the classical Legendre formula:  
$v_p(n!)=\sum_{\ell=1}^\infty\left\lfloor n/p^\ell\right\rfloor$.  
Surprisingly, certain basic properties of the Landau function $\Delta_{e,f}$ turn out to characterize fundamental 
arithmetic properties of the corresponding factorial ratio 
and its generating series. Indeed, assuming for simplicity that $\sum_i e_i=\sum_j f_j$, we have the following results. 
\begin{itemize}
\item [(i)] The sequence $(Q_{e,f}(n))_{n\geq 0}$ takes integer values if and only if $\Delta_{e,f}(x)\geq0$, $\forall x\in[0,1]$.
\item [(ii)] The sequence $(Q_{e,f}(n))_{n\geq 0}$ has the $p$-Lucas property for all primes $p$
\footnote{That is $Q_{e,f}(pn+r)\equiv Q_{e,f}(n)Q_{e,f}(r)\mod p$ for every $r\in \{0,\ldots,p-1\}$ and $n\geq 0$.}  if and only if $\Delta_{e,f}(x)\geq1$, $\forall x\in[m_{e,f},1]$, 
where $m_{e,f}:=1/\mbox{max}\{e_1,\ldots,e_v,f_1,\ldots,f_w\}$.
\item [(iii)] The generating series of $(Q_{e,f}(n))_{n\geq 0}$ is algebraic\footnote{This means that  the power series $\sum_{n=0}^\infty Q_{e,f}(n)x^n\in\mathbb Q[[x]]$ is algebraic over the field $\mathbb Q(x)$.} if and only if $\Delta_{e,f}(x)\in\{0,1\}$, $\forall x\in[0,1]$.
\end{itemize}

Items (i) and (iii) were respectively proved by Landau \cite{Landau} (see also \cite{Bober}) and Rodriguez-Villegas \cite{R-V} (as a consequence of \cite{BH}). Item (ii) 
corresponds to \cite[Proposition 8.3]{ABD1} and was derived from \cite[Theorem 3]{DelaygueAp}. 

Choosing for example $e=(30,1)$ and $f=(15,10,6)$, a straightforward computation shows that the corresponding sequence takes integer values, 
does not have the $p$-Lucas property for all primes, and has an algebraic generating series. At first sight, proving this result is not easy:    
for example, Rodriguez-Villegas \cite{R-V} observed that the degree of algebraicity is $483\, 840$. 

\medskip

These results have been generalized, replacing factorials by Pochhammer symbols and factorial ratios by hypergeometric sequences. 
We recall that the Pochhammer symbol $(x)_n$, also called rising factorial, is defined as 
$$(x)_n=x(x+1)\cdots(x+n-1)\,,$$ 
if $n\geq 1$ and $(x)_0=1$, so that $(1)_n=n!$ and 
\begin{equation}\label{eq: dn!} 
(dn)!=d^{dn}\left(\frac{1}{d}\right)_n\cdots\left(\frac{d-1}{d}\right)_n(1)_n\,.
\end{equation}
Given $\alpha\in\mathbb Q\setminus\mathbb Z_{\leq 0}$ and $p$ a prime such that $v_p(\alpha)\geq 0$, 
Christol \cite{Christol} provided the following formula\footnote{More exactly, Formula \eqref{eq:vpPocChristol} is a reformulation 
with floor functions of Christol's result, as given  in \cite[Section~5.3]{DRR}.} 
 for the $p$-adic valuation of Pochhammer symbols:
\begin{equation}\label{eq:vpPocChristol}
v_p((\alpha)_n)=\sum_{\ell=1}^\infty\left\lfloor\frac{n-\lfloor 1-\alpha\rfloor}{p^\ell}-D_{p}^\ell(\alpha)+1\right\rfloor \,,
\end{equation}
where $D_{p}(\alpha)$ is defined as the unique rational number whose denominator is not divisible by $p$ and such that $p D_{p}(\alpha)-\alpha$ belongs to $\{0,\dots,p-1\}$. The maps  
$\alpha\mapsto D_{p}(\alpha)$ were first introduced by Dwork \cite{Dwork} and are now referred to as Dwork maps.  
When $\alpha=1$, we have $D_{p}(1)=1$ and we retrieve Legendre's formula. Note also that  if $v_p(\alpha)<0$, then 
we simply have $v_p((\alpha)_n)=nv_p(\alpha)$. 

\medskip

Given two vectors $\boldsymbol{\alpha}:=(\alpha_1,\ldots,\alpha_v)$ and $\boldsymbol{\beta}:=(\beta_1,\ldots,\beta_w)$,
whose coordinates belong to $\mathbb Q\setminus \mathbb Z_{\leq 0}$, we define the (generalized) \emph{hypergeometric sequence}  
\begin{equation}\label{eq: hypseq}
Q_{\boldsymbol{\alpha},\boldsymbol{\beta}}(n) := \frac{(\alpha_1)_n\cdots (\alpha_v)_n}{(\beta_1)_n\cdots(\beta_w)_n} \in \mathbb Q \quad\quad n\geq 0 \,.
\end{equation}
The above restriction on the rational parameters  $\beta_j$ ensures that $Q_{\boldsymbol{\alpha},\boldsymbol{\beta}}(n)$ is well-defined for all $n\geq 0$. We also assume that the 
parameters $\alpha_i$ do not belong to $\mathbb Z_{\leq 0}$, since otherwise $Q_{\boldsymbol{\alpha},\boldsymbol{\beta}}(n)$ 
would vanish for all $n$ large enough, which would be irrelevant for our purpose. 
These sequences and their generating series have attracted a lot of attention since the time of Gauss. 
According to  \eqref{eq: dn!}, the study of factorial ratios reduces to the study of certain hypergeometric sequences. 
Again, understanding how arithmetic properties of hypergeometric sequences may depend on the rational parameters $\alpha_i$ and $\beta_j$ leads to fascinating questions. 

We let $d_{\boldsymbol{\alpha},\boldsymbol{\beta}}$ denote the least common multiple of the denominators of the parameters $\alpha_i$ and $\beta_j$. 
 In \cite{Christol}, Christol  introduced new step functions $\xi_{\boldsymbol{\alpha},\boldsymbol{\beta}}(a,\cdot)$, for every $a\in \{1,\ldots,d_{\boldsymbol{\alpha},\boldsymbol{\beta}}\}$ coprime to $d_{\boldsymbol{\alpha},\boldsymbol{\beta}}$, which play the same role for hypergeometric sequences as the Landau function $\Delta_{e,f}$ does 
  for factorial ratios. We refer the reader to Section \ref{sec:fcChristol} for a definition.  

Analogs of (i)--(iii) have been respectively obtained by Christol \cite{Christol}, Adamczewski, Bell, and Delaygue \cite{ABD1}, and 
Beukers and Heckman \cite{BH}\footnote{We 
 refer the reader to \cite{Christol,ABD1,BH} for precise statements. The reformulation in terms of Christol step functions of the famous interlacing criterion of Beukers and Heckmann can be found in \cite{DRR}.}.   We point out that, for the analog of (i),  
it is more natural to consider $N$-integrality of the sequence $(Q_{\boldsymbol{\alpha},\boldsymbol{\beta}}(n))_{n\geq 0}$, that is to ask whether there exists a non-zero 
integer $a$ such that $a^nQ_{\boldsymbol{\alpha},\boldsymbol{\beta}}(n)\in \mathbb Z$ for all $n\geq 0$. Also, for the analog of (ii), it is more natural to consider the $p$-Lucas 
property for all but finitely many primes in a given residue class  modulo $d_{\boldsymbol{\alpha},\boldsymbol{\beta}}$.  
Finally, the required conditions about the Landau function must now be satisfied by the Christol functions $\xi_{\boldsymbol{\alpha},\boldsymbol{\beta}}(a,\cdot)$ for all $a\in \{1,\ldots,d_{\boldsymbol{\alpha},\boldsymbol{\beta}}\}$ coprime to $d_{\boldsymbol{\alpha},\boldsymbol{\beta}}$. 
In particular, the analog of (i) proved by Christol~\cite{Christol} reads as follows. 

\begin{thmC} Let $\boldsymbol{\alpha}:=(\alpha_1,\ldots,\alpha_u)$ and $\boldsymbol{\beta}:=(\beta_1,\ldots,\beta_v)$ be two vectors 
whose coordinates belong to $\mathbb Q\setminus \mathbb Z_{\leq 0}$. Then the two following assertions are equivalent.

\begin{itemize}
\item[{\rm (a)}] The hypergeometric sequence $(Q_{\boldsymbol{\alpha},\boldsymbol{\beta}}(n))_{n\geq 0}$ is $N$-integral.

\item[{\rm (b)}]  For every $a$ in $\{1,\dots,d_{\boldsymbol{\alpha},\boldsymbol{\beta}}\}$ coprime to $d_{\boldsymbol{\alpha},\boldsymbol{\beta}}$ 
and all $x$ in $\mathbb{R}$, we have $\xi_{\boldsymbol{\alpha},\boldsymbol{\beta}}(a,x)\geq 0$.
\end{itemize}
 \end{thmC}
 
A remarkable feature of (i)--(iii) and of the results proved in  \cite{Christol,ABD1,BH} is  that they provide simple algorithms, given in terms of suitable step functions, 
that allow one to decide whether certain fundamental arithmetic properties of factorial ratios and hypergeometric sequences hold \footnote{We also refer the reader to \cite{AS} for more general results about integrality  of $A$-hypergeometric series.}.

\subsection{Main results}\label{sec:mainresults}

In this paper, our main objective is to prove $q$-analogs of Formula \eqref{eq:vpPocChristol} and Theorem C. 
From now on, we let $q$ denote a fixed transcendental complex number. 

We are going to define suitable $q$-analogs of the Pochhammer symbol $(\alpha)_n$ and of the hypergeometric term 
$Q_{\boldsymbol{\alpha},\boldsymbol{\beta}}(n)$, which belong to the field $\mathbb Q(q)$. In this framework, 
the $p$-adic valuations are replaced by the cyclotomic valuations, while   
the notion of $N$-integrality is replaced by $q$-integrality. 
For every positive integer $b$, we let $\phi_b(q)\in\mathbb Z[q]$ denote the $b$th cyclotomic polynomial and $v_{\phi_b}$ stands for 
the valuation of $\mathbb Q(q)$ associated with $\phi_b(q)$ (see Section \ref{subsec: cyclo} for a definition).  
A sequence $(R(q;n))_{n\geq 0}$ with values in $\mathbb{Q}(q)$ and first term $R(q;0)=1$ is said to be \textit{$q$-integral} 
if there exists  $C(q)\in\mathbb{Z}[q]\setminus\{0\}$ such that $C(q)^nR(q;n)\in\mathbb{Z}[q]$ for all $n\geq 0$. 

\medskip

For every positive integer $n$, the $q$-analog of the integer $n$ is defined as $[n]_q=1+q+\cdots+q^{n-1}$, while $[0]_q=0$. It is actually convenient to write 
$$
[n]_q=\frac{1-q^n}{1-q}\,,
$$
while keeping in mind that this ratio belongs to $\mathbb{Z}[q]$. It follows that 
$$
[n]_q=\prod_{b\geq 2,\,b\mid n}\phi_b(q)\,,
$$
which specializes as 
\begin{equation}\label{eq: simple fact}
n=\prod_{b\geq 2,\,b\mid n}\phi_b(1)\,.
\end{equation}
We recall that $\phi_b(1)=1$ if $b$ is divisible by at least two distinct primes, while $\phi_{p^\ell}(1)=p$ when $p$ is a prime and $\ell$ is a positive integer. 
We deduce that 
\begin{equation}\label{eq:LegendrePhib}
v_p(n)=\sum_{\ell=1}^\infty v_{\phi_{p^\ell}}([n]_q)\,.
\end{equation}
This formula shows that, in some sense, the arithmetic of $q$-analogs is finer than that of integers.
The $q$-analog of $n!$ is defined as 
$$
[n]!_q:=\prod_{i=1}^n\frac{1-q^i}{1-q}\,\cdot
$$
Given $\alpha=r/s$ a rational number,  the $q$-analog of the Pochhammer symbol $(\alpha)_n$ is  most often defined as (see, for instance, \cite{GR}) 
$$
\frac{(q^\alpha;q)_n}{(1-q)^n}\in \mathbb Q(q^{1/s})\,,
$$
where we let $(a;q)_n:=\prod_{i=0}^{n-1}(1-aq^i)$ denote the $q$-Pochhammer symbol (also called the $q$-shifted factorial).  
Substituting $q$ by $q^s$, we obtain a slightly different $q$-analog of $(\alpha)_n$: 
\begin{equation}\label{eq: qalpha}
\frac{(q^r;q^s)_n}{(1-q^s)^n}\in \mathbb Q(q)\,.
\end{equation}
We note that 
$$
\lim_{q\to 1}\frac{(q^\alpha;q)_n}{(1-q)^n} = \lim_{q\to 1}\frac{(q^r;q^s)_n}{(1-q^s)^n}=(\alpha)_n\,.
$$
The latter has several advantages which are discussed in Section \ref{sec: choice}.  
In the end, it is sufficient for our discussion to consider $q$-Pochhammer symbols of the form 
$$
(q^r;q^s)_n:= \prod_{i=0}^{n-1}(1-q^{r+si})\in\mathbb{Z}[q^{-1},q]\,,
$$
where $r$ and $s$ are two integers, $s\not=0$.  
This product is non-zero if and only if $r/s\notin\mathbb{Z}_{\leq 0}$ or $n\leq -r/s$. The usual extension to negative arguments $n$ is given by 
\begin{equation}\label{eq: hgneg}
(q^r;q^s)_n=\prod_{i=1}^{-n}\frac{1}{(1-q^{r-is})}=\frac{1}{(q^{r-s};q^{-s})_{-n}}\, ,
\end{equation}
which is well-defined if and only if $r/s\notin\mathbb{Z}_{>0}$ or $n> -r/s$. 

\medskip

Our first main result, which provides a $q$-analog of Formula \eqref{eq:vpPocChristol} as well as its extension to negative arguments, 
involves a generalization of Dwork maps where the prime number $p$ is replaced 
by an arbitrary positive integer $b$. 
Given a positive integer $b$ and a rational number $\alpha$ whose denominator is coprime to $b$, we show in Section \ref{sec: dwork} that there exists a unique rational number 
$D_b(\alpha)$ whose denominator is coprime to $b$ and which satisfies $bD_b(\alpha)-\alpha\in\{0,\dots,b-1\}$. 
When $b=p$ is prime, we retrieve the classical Dwork map $D_p$.

\begin{thm}\label{thm: valuationq-pocch}
Let $r$ and $s$ be two integers, $s\neq 0$, and $\alpha:=r/s$. Let $b$ be a positive integer, $c:=\gcd(r,s,b)$, $b':=b/c$, and $s':=s/c$. 
Let $n\in\mathbb Z$ be  such that $(q^r;q^s)_n$ is well-defined and non-zero. Then we have 
$$
v_{\phi_b}((q^r;q^s)_n)=\left\{\begin{array}{cl}
\left\lfloor cn/b-  D_{b'}\left(\alpha\right)-\frac{\lfloor 1-\alpha\rfloor}{b'}        \right\rfloor+1 & \textup{if $\gcd(s',b')=1$ \,,}\\
0 & \textup{otherwise.}
\end{array}\right.
$$
\end{thm}

\begin{rem}\emph{
Recall that $v_{\phi_b}((1-q^s)^n)=nv_{\phi_b}(1-q^s)$ and $v_{\phi_b}(1-q^s)=1$ if $b$ divides $s$ and $0$ otherwise.  Hence   
we can  easily derive from Theorem \ref{thm: valuationq-pocch} a formula for the $\phi_b$-valuation of the $q$-analog of $(\alpha)_n$ given in  
\eqref{eq: qalpha}.}
\end{rem}

We now define $q$-analogs of hypergeometric sequences with rational parameters. 
For $i\in\{1,\dots,v\}$ and $j\in\{1,\dots,w\}$, we let $(r_i,s_i)$ and $(t_j,u_j)$ be pairs of integers 
such that $s_i\not=0$ and $u_j\not=0$. We set 
$$
\mathbf{r}:=((r_1,s_1),\dots,(r_v,s_v))\quad\textup{and}\quad \mathbf{t}:=((t_1,u_1),\dots,(t_w,u_w))\, ,
$$
together with $\boldsymbol{\alpha}:=(\alpha_1,\dots,\alpha_v)$ and $\boldsymbol{\beta}:=(\beta_1,\dots,\beta_w)$, where $\alpha_i:=r_i/s_i$ and $\beta_j:=t_j/u_j$. 
Let $d_{\mathbf{r},\mathbf{t}}:=\lcm\{s_1,\dots,s_v,u_1,\dots,u_w\}$. 
With this notation, we define the \emph{$q$-hypergeometric sequence}  
\begin{equation}\label{qPocchRatios}
Q_{\mathbf{r},\mathbf{t}}(q;n):=\frac{(q^{r_1};q^{s_1})_n\cdots(q^{r_v};q^{s_v})_n}{(q^{t_1};q^{u_1})_n\cdots(q^{t_w};q^{u_w})_n} \quad\quad n\geq 0\,.
\end{equation}
Note that, similarly to \eqref{eq: hypseq},  
$Q_{\mathbf{r},\mathbf{t}}(q;n)$ is well-defined for all $n\geq 0$ when the rational numbers $\beta_j$ do not belong to $\mathbb Z_{\leq 0}$. 
In addition, we assume that the rational numbers $\alpha_i$ do not belong to $\mathbb Z_{\leq 0}$, since otherwise $Q_{\mathbf{r},\mathbf{t}}(q;n)$ 
would vanish for all $n$ large enough, which would be irrelevant for our purpose.

Our second main result is a $q$-analog of Theorem C. It involves new step functions $\Xi_{\mathbf{r},\mathbf{t}}(b,\cdot)$,  
$b\in\{1,\ldots,d_{\mathbf{r},\mathbf{t}}\}$, which generalize Christol step functions. They are introduced in Section~\ref{sec: effi}, where we also show that 
$\Xi_{\mathbf{r},\mathbf{t}}(b,\cdot)=\xi_{\boldsymbol{\alpha},\boldsymbol{\beta}}(a,\cdot)$ for $b$ coprime to $d_{\mathbf{r},\mathbf{t}}$ and  
 $ba\equiv 1\bmod d_{\mathbf{r},\mathbf{t}}$.   
Thus, we only define new functions for $b$ not coprime to $d_{\mathbf{r},\mathbf{t}}$. 
The appearance of these new functions makes the proof of Theorem~\ref{thm:N-integralityPolyn} substantially more tricky 
than that of Theorem~C. 

\begin{thm}\label{thm:N-integralityPolyn} We continue with the previous notation and assumptions.  We also assume that $s_1,\ldots,s_v$ are positive.  
Then the two following assertions are equivalent.
\begin{itemize}
\item[$\mathrm{(i)}$] The sequence $(Q_{\mathbf{r},\mathbf{t}}(q; n))_{n\geq 0}$ is $q$-integral.
\item[$\mathrm{(ii)}$] For every $b\in\{1,\dots,d_{\mathbf{r},\mathbf{t}}\}$ and all $x$ in $\mathbb{R}$, we have $\Xi_{\mathbf{r},\mathbf{t}}(b,x)\geq 0$.
\end{itemize}
\end{thm}

A generalization of Theorem \ref{thm:N-integralityPolyn} with no restriction on the parameters $s_1,\dots,s_v\in\mathbb{Z}\setminus\{0\}$ is stated as Theorem \ref{thm:N-integrality} in Section \ref{sec:Efficient}.

\begin{rem}\emph{Strictly speaking, $Q_{\mathbf{r},\mathbf{t}}(q;n)$ is not a $q$-analog of the hypergeometric term 
$\mathcal{Q}_{\boldsymbol{\alpha}, \boldsymbol{\beta}}(n)$. Instead, \eqref{eq: qalpha} shows that a suitable $q$-analog  can be defined as  
$$
 \mathcal{Q}'_{\mathbf{r},\mathbf{t}}(q; n):= \left(\frac{\prod_{j=1}^w(1-q^{u_j}) }{\prod_{i=1}^v (1-q^{s_i})}\right)^{n} \mathcal{Q}_{\mathbf{r},\mathbf{t}}(q; n)\,.
$$
Indeed, we have 
\begin{equation}\label{eq: conf}
 \underset{q\to 1}{\lim}\,\mathcal{Q}'_{\mathbf{r},\mathbf{t}}(q; n)=\mathcal{Q}_{\boldsymbol{\alpha}, \boldsymbol{\beta}}(n)\,.
\end{equation}
Since the $q$-integrality of $(\mathcal{Q}_{\mathbf{r},\mathbf{t}}(q; n))_{n\geq 0}$ is equivalent to that of  
$(\mathcal{Q}'_{\mathbf{r},\mathbf{t}}(q; n))_{n\geq 0}$, we find more convenient to work with the simpler expression $\mathcal{Q}_{\mathbf{r},\mathbf{t}}(q; n)$.} 
 \end{rem}

We infer from \eqref{eq: conf} that the $q$-integrality of the sequence $(Q_{\mathbf{r},\mathbf{t}}(q;n))_{n\geq 0}$ implies the $N$-integrality of the sequence $(\mathcal{Q}_{\boldsymbol{\alpha}, \boldsymbol{\beta}}(n))_{n\geq 0}$. This is consistent with Theorems \ref{thm:N-integralityPolyn} and C since 
$\Xi_{\mathbf{r},\mathbf{t}}(b,\cdot)=\xi_{\boldsymbol{\alpha},\boldsymbol{\beta}}(a,\cdot)$ when $ba\equiv 1\bmod d_{\mathbf{r},\mathbf{t}}$. 
 However, the converse result does not always hold true, depending on the behaviour of  $\Xi_{\mathbf{r}, \mathbf{t}}(b,\cdot)$ for $b$ not coprime to $d_{\mathbf{r}, \mathbf{t}}$. 

For example, let us consider the vectors 
$$
\mathbf{r}:=((1,3),(2,3))\quad\textup{and}\quad\mathbf{t}:=((1,2),(1,1))\,.
$$
Then we have $\boldsymbol{\alpha}=(1/3,2/3)$ and $\boldsymbol{\beta}=(1/2,1)$. We deduce from  \eqref{eq: dn!} (or from Theorem C) that the hypergeometric sequence 
$$
\mathcal{Q}_{\boldsymbol{\alpha}, \boldsymbol{\beta}}(n)=\frac{(1/3)_n(2/3)_n}{(1/2)_n(1)_n} \quad\quad n\geq0\,
$$
is $N$-integral. However, we have $\Xi_{\mathbf{r},\mathbf{t}}(3,1/2)<0$ (see Section~\ref{sec:general} for more details)    
and thus, according to Theorem~\ref{thm:N-integralityPolyn}, the $q$-hypergeometric sequence 
$$
\mathcal{Q}_{\mathbf{r},\mathbf{t}}(q; n)=\frac{(q; q^3)_n(q^2;q^3)_n}{(q;q^2)_n(q;q)_n} \quad\quad n\geq 0\,
$$ 
fails to be $q$-integral.

\subsection{Organization of the paper}

In Section \ref{sec: choice}, we discuss our choice for the $q$-analog of the Pochhammer symbol $(\alpha)_n$ and show how to relate our results on 
$q$-hypergeometric sequences to basic hypergeometric series, as they are usually defined. 
In Section \ref{sec: val}, we extend the definition of Dwork maps to arbitrary integers $b$ and prove some of their basic properties.  
We also prove Theorem \ref{thm: valuationq-pocch}, as well as a formula for the cyclotomic valuation of $q$-hypergeometric terms. 
The latter is given in terms of certain step functions $\Delta_b^{\mathbf{r},\mathbf{t}}$, which are introduced in this section.  
In Section \ref{sec:CriterionIntegrality}, we deduce a first criterion for the $q$-integrality of 
$q$-hypergeometric sequences, which depends on the behaviour of $\Delta_b^{\mathbf{r},\mathbf{t}}$ for all but finitely many integers $b$. 
We also discuss the extension of this result to negative arguments $n$.  
These first criteria for $q$-integrality are not very satisfactory because they imply checking certain properties of an infinite number of step functions. 
We fill this gap in Section \ref{sec: effi}, where we introduce the finitely many step functions 
$\Xi_{\mathbf{r},\mathbf{t}}(b,\cdot)$, $b\in\{1,\ldots,d_{\mathbf{r},\mathbf{t}}\}$, and prove Theorem \ref{thm:N-integralityPolyn}.  
Finally, we provide some illustrations of Theorem \ref{thm:N-integralityPolyn} in Section \ref{sec: ex}.

\section{Choices for the $q$-analogs of Pochhammer symbols and hypergeometric functions}\label{sec: choice}

The notion of $q$-analog is loosely defined: for $a(q)$ to be a $q$-analog of a term $a$, one only requires that $a(q)$ tends to $a$ as $q$ tends to $1$. 
While everyone agrees with the definition of $[n]_q$ and $[n]!_q$, this requires a fair amount of choice for more general expressions. 
Depending on the nature of the properties one wishes to study, one may have to make one choice rather than another. 
In this section, we discuss in more detail our own choices for the $q$-analogs of Pochhammer symbols and hypergeometric series, as well as how 
our results translate when considering other natural $q$-analogs.

\subsection{Cyclotomic valuations and $q$-valuation}\label{subsec: cyclo}

We recall that, for every positive integer $b$, $\phi_b(q)\in\mathbb Z[q]$ stands for the $b$th cyclotomic polynomial. It is well-known that $\phi_b(q)$ is irreducible over 
$\mathbb Z[q]$. If $R$ and $S$ belong to 
$\mathbb{Z}[q]\setminus\{0\}$, 
then we let $v_{\phi_b}(R)$ denote the ${\phi_b}$-valuation of $R$, that is the largest non-negative integer $\nu$ such that $\phi_b(q)^\nu$ divides $R$. 
We also set $v_{\phi_b}(0):=+\infty$. 
The $\phi_b$-valuation extends naturally to $\mathbb Q(q)$ by setting $v_{\phi_b}(R/S):=v_{\phi_b}(R)-v_{\phi_b}(S)$. 

We also let $v_q$ denote the valuation of $\mathbb Q(q)$ which is associated with the irreducible polynomial $q$ in the same way.

\subsection{$q$-Analogs of Pochhammer symbols}

We explain now why we prefer to choose 
\begin{equation}\label{eq: ochoice}
\frac{(q^r;q^s)_n}{(1-q^s)^n}\in \mathbb Q(q)\,
\end{equation}
as $q$-analog of the Pochammer symbol $(\alpha)_n$,  $\alpha=r/s$, instead of the more standard 
\begin{equation}\label{eq: schoice}
\frac{(q^\alpha;q)_n}{(1-q)^n}\in \mathbb Q(q^{1/s})\,.
\end{equation}
There are three main reasons for our preference. The first one, which was already mentioned in the introduction,  
is that we find it more natural to work in the  field $\mathbb Q(q)$ instead of working in the field $\cup_{s\geq 1}\mathbb Q(q^{1/s})$ and dealing with non-integer powers of $q$. 
The second one is that it offers more flexibility. For example,  
$$
\frac{(q;q^2)_n}{(1-q^2)^n}\,,\quad \frac{(q^3;q^6)_n}{(1-q^6)^n} \,,\quad \text{and } \quad \frac{(q^{-1};q^{-2})_n}{(1-q^{-2})^n} 
$$
provide three different $q$-analogs of $(1/2)_n$. 
The third one comes from the useful 
Equality~\eqref{eq: dn!}, which we recall here for the reader's convenience: 
\begin{equation}\label{eq: dn!bis}
(dn)!=d^{dn}\left(\frac{1}{d}\right)_n\left(\frac{2}{d}\right)_n\cdots\left(\frac{d-1}{d}\right)_n(1)_n\,.
\end{equation}
With the choice of $(q^\alpha;q)_n/(1-q)^n$, we do not obtain a nice $q$-deformation of \eqref{eq: dn!bis}.  
Indeed, take for instance $d=2$, so that  
$$(2n)!=4^n(1/2)_n(1)_n\,.$$
 The $q$-analog of the left-hand side of \eqref{eq: dn!bis} is  
 $$
 \frac{(q;q)_{2n}}{(1-q)^{2n}}=\frac{(q;q^2)_{n}(q^2;q^2)_{n}}{(1-q)^{2n}}=(-q^{1/2};q)_{n}(-q;q)_{n}\frac{(q^{1/2};q)_{n}}{(1-q)^{n}}\frac{(q;q)_n}{(1-q)^{n}}\,,
 $$
therefore introducing minus signs in $q$-Pochhammer symbols. In contrast, the choice $(q^r;q^s)_n/(1-q^s)^n$ ensures the following nice $q$-deformation of \eqref{eq: dn!bis}: 
$$
[dn]!_q=\prod_{i=1}^{dn}\frac{1-q^i}{1-q}=\left(\frac{1-q^d}{1-q}\right)^{dn}\prod_{i=1}^{d}\frac{(q^i;q^d)_n}{(1-q^d)^n}\,\cdot
$$

\begin{rem}\label{rem: iso}\emph{ Let $d$ be a positive integer. Since $q$ is transcendental over $\mathbb{Q}$, 
there is an isomorphism of $\mathbb{Z}$-modules given by
$$
\begin{array}{rlll}
\varphi: & \mathbb{Z}[q^{1/d}] & \longrightarrow & \mathbb{Z}[q]\\
& P(q^{1/d}) & \longmapsto & P(q)
\end{array}.
$$
In particular, $\mathbb{Z}[q^{1/d}]$ is a Euclidean ring whose irreducible elements are of the form $P(q^{1/d})$ where $P(q)$ is an irreducible polynomial in $\mathbb{Z}[q]$. The isomorphism $\varphi$ extends to an isomorphism between the rings of Laurent polynomials $\mathbb{Z}[q^{-1/d},q^{1/d}]$ and $\mathbb{Z}[q^{-1},q]$, as well as 
between the  fields $\mathbb{Q}(q^{1/d})$ and $\mathbb{Q}(q)$. In particular, 
if we let $v_{b,s}$ denote the valuation in $ \mathbb Q(q^{1/s})$ associated with the irreducible polynomial 
$\phi_b(q^{1/s})\in\mathbb Z[q^{1/s}]$ and if we take $\alpha=r/s$, then we obtain that }
$$v_{b,s}((q^\alpha;q)_n/(1-q)^n)=v_{\phi_b}((q^r;q^s)_n/(1-q^s)^n)\,.$$ 
\emph{This shows that there is no loss of generality when choosing \eqref{eq: ochoice} as  
$q$-analog of $(\alpha)_n$.  }
\end{rem}

\subsection{$q$-Analogs of generalized hypergeometric series}

Let us first recall the standard notation for hypergeometric series (with rational parameters). With two vectors $\boldsymbol{\alpha}=(\alpha_1,\dots,\alpha_v)$ and 
$\boldsymbol{\beta}=(\beta_1,\dots,\beta_w)$ whose coordinates belong to $\mathbb{Q}\setminus\mathbb{Z}_{\leq 0}$, we associate the generalized hypergeometric series 
defined by 
$$
{}_{v}F_w\left(\begin{array}{cccc}\alpha_1,&\dots&,\alpha_v\\ \beta_1,&\dots&,\beta_w&\end{array}\Bigg|\,x\right):=\sum_{n=0}^\infty\frac{(\alpha_1)_n\cdots(\alpha_v)_n}{(\beta_1)_n\cdots(\beta_w)_n n!}x^n \,,
$$
while we usually prefer to work with its companion power series 
$$
F_{\boldsymbol{\alpha},\boldsymbol{\beta}}(x):= 
{}_{v+1}F_w\left(\begin{array}{ccccc}\alpha_1,&\dots&,\alpha_v,&1\\ \beta_1,&\dots&,\beta_w&\end{array}\Bigg|\,x\right)
=\sum_{n=0}^\infty\frac{(\alpha_1)_n\cdots(\alpha_v)_n}{(\beta_1)_n\cdots(\beta_w)_n}x^n\,.
$$

The \emph{basic hypergeometric series} is defined as 
\begin{equation*}
{}_{v}\phi_w\left(\begin{array}{cccc}q^{\alpha_1},&\dots&,q^{\alpha_v}\\ q^{\beta_1},&\dots&,q^{\beta_w}&\end{array}\Bigg|\,q;x\right):=\sum_{n=0}^\infty\frac{(q^{\alpha_1};q)_n\cdots(q^{\alpha_v};q)_n}{(q^{\beta_1};q)_n\cdots(q^{\beta_w};q)_n(q;q)_n}\left((-1)^nq^{\binom{n}{2}}\right)^{1+w-v}x^n\,.
\end{equation*} 
It is a generalization of the classical  ${}_{2}\phi_1$ introduced by Heine \cite{He} and the most standard $q$-analog of the hypergeometric series 
${}_{v}F_w$ (see, for instance, the monograph~\cite{GR} for more  on this topic). 
In fact, it is a $q$-analog up to renormalization by a factor $(q-1)^{(w-v)n}$, that is 
$$
\underset{q\rightarrow 1}{\lim}\,{}_{v+1}\phi_w\left(\begin{array}{cccc}q^{\alpha_1},&\dots&,q^{\alpha_v},&q\\ q^{\beta_1},&\dots&,q^{\beta_w}&\end{array}\Bigg|\,q;(q-1)^{w-v}x\right)=F_{\boldsymbol{\alpha},\boldsymbol{\beta}}(x)\,.
$$
Hence a first $q$-analog of $F_{\boldsymbol{\alpha},\boldsymbol{\beta}}(x)$ is given by 
\begin{align}\label{eq: F1}
F^{(1)}_{\boldsymbol{\alpha},\boldsymbol{\beta}}(q;x) :&= {}_{v+1}\phi_w\left(\begin{array}{cccc}q^{\alpha_1},&\dots&,q^{\alpha_v},&q\\ q^{\beta_1},&\dots&,q^{\beta_w}&\end{array}\Bigg|\,q; (q-1)^{w-v}x\right) \\
\nonumber&\displaystyle=\sum_{n=0}^\infty\frac{(q^{\alpha_1};q)_n\cdots(q^{\alpha_v};q)_n}{(q^{\beta_1};q)_n\cdots(q^{\beta_w};q)_n}\cdot (1-q)^{(w-v)n}q^{(w-v)\binom{n}{2}}x^n  \,.
\end{align}

Now, choosing $(q^\alpha;q)_n/(1-q)^n$ as $q$-analog of the Pochhammer symbol $(\alpha)_n$, we obtain another natural $q$-analog of 
$F_{\boldsymbol{\alpha},\boldsymbol{\beta}}(x)$, namely
\begin{equation}\label{eq: F2}
F^{(2)}_{\boldsymbol{\alpha},\boldsymbol{\beta}}(q;x):=\sum_{n=0}^\infty\frac{(q^{\alpha_1};q)_n\cdots(q^{\alpha_v};q)_n}{(q^{\beta_1};q)_n\cdots(q^{\beta_w};q)_n}
\cdot (1-q)^{(w-v)n}x^n \,.
\end{equation}
The two definitions only differ by the factor $q^{(w-v)\binom{n}{2}}$. In particular, they coincide when $v=w$.  

Finally, choosing $(q^r;q^s)_n/(1-q^s)^n$ as $q$-analog of the Pochhammer symbol $(\alpha)_n$, $\alpha=r/s$, we obtain a third natural $q$-analog of 
$F_{\boldsymbol{\alpha},\boldsymbol{\beta}}(x)$, namely
\begin{equation}\label{eq: F3}
F_{\boldsymbol{r},\boldsymbol{t}}(q;x):=\sum_{n=0}^\infty\frac{(q^{r_1};q^{s_1})_n\cdots(q^{r_v};q^{s_v})_n}{(q^{t_1};q^{u_1})_n\cdots(q^{t_w};q^{u_w})_n}\cdot
  \left(\frac{(1-q^{u_1})\cdots(1-q^{u_w})}{(1-q^{s_1})\cdots(1-q^{s_v})}\right)^n\; x^n,
\end{equation}
where $\boldsymbol{r}=((r_1,s_1),\ldots,(r_v,s_v))$, $\boldsymbol{t}=((t_1,u_1),\ldots,(t_w,u_w))$, and $r_i/s_i=\alpha_i$ and $t_j/u_j=\beta_j$ for all $i$ and $j$.

\medskip

Thus, we have three different natural $q$-analogs of the generalized hypergeometric series $F_{\boldsymbol{\alpha},\boldsymbol{\beta}}(x)$.  
We observe that both $F^{(1)}_{\boldsymbol{\alpha},\boldsymbol{\beta}}(q;x)$ and $F^{(2)}_{\boldsymbol{\alpha},\boldsymbol{\beta}}(q;x)$ 
have coefficients in $\mathbb{Q}(q^{1/d})$, where $d=d_{{\boldsymbol{\alpha},\boldsymbol{\beta}}}$ is the least common multiple of the denominators of the rational numbers 
$\alpha_i$ and $\beta_j$.  
In contrast, $F_{\boldsymbol{r},\boldsymbol{t}}(q;x)$ has coefficients in $\mathbb Q(q)$ and there exist infinitely many 
vectors $\boldsymbol{r}$ and $\boldsymbol{t}$ such that 
$$
\lim_{q\to 1}F_{\boldsymbol{r},\boldsymbol{t}}(q;x) = F_{\boldsymbol{\alpha},\boldsymbol{\beta}}(x)\,.
$$
Indeed, if $\boldsymbol{r}=((r_1,s_1),\ldots,(r_v,s_v))$ and $\boldsymbol{t}=((t_1,u_1),\ldots,(t_w,u_w))$ is such a pair of vectors, then for each  
pair $(a,b)$ occurring either in $\boldsymbol{r}$ or in  $\boldsymbol{t}$, we can choose a non-zero integer $k$ and replace $(a,b)$ by $(ka,kb)$.

\subsection{$q$-Integrality and $q^{1/d}$-integrality for basic hypergeometric series}

A power series $F(q;x)\in 1+x\mathbb Q(q)[[x]]$ is said to be $q$-integral if the sequence formed by its coefficients is $q$-integral, or, in other words, 
if there exists $C(q)\in\mathbb Z[q]\setminus\{0\}$ such that $F(q;C(q)x)\in\mathbb Z[q][[x]]$. 

Similarly, we say that a power series 
$F(q;x)\in 1+x\mathbb Q(q^{1/d})[[x]]$ is $q^{1/d}$-integral if there exists $C(q)\in\mathbb Z[q^{1/d}]\setminus\{0\}$ such that 
$F(q;C(q)x)\in\mathbb Z[q^{1/d}][[x]]$. According to Remark \ref{rem: iso}, $F(q;x)$ is $q^{1/d}$-integral if and only if  
$F(q^d;x)$ is $q$-integral.

Now, we show how Theorem \ref{thm:N-integralityPolyn}  can be used to study the $q^{1/d}$-integrality of $F^{(1)}_{\boldsymbol{\alpha},\boldsymbol{\beta}}(q;x)$ and 
$F^{(2)}_{\boldsymbol{\alpha},\boldsymbol{\beta}}(q;x)$, as well as the $q$-integrality of $F_{\boldsymbol{r},\boldsymbol{t}}(q;x)$.  
Recall that  
\begin{equation*}
F^{(1)}_{\boldsymbol{\alpha},\boldsymbol{\beta}}(q^d;x)=\sum_{n=0}^\infty\frac{(q^{d\alpha_1};q^d)_n\cdots(q^{d\alpha_v};q^d)_n}{(q^{d\beta_1};q^d)_n\cdots(q^{d\beta_w};q^d)_n}\cdot (1-q^d)^{(w-v)n}q^{d(w-v)\binom{n}{2}}x^n 
\end{equation*}
and 
\begin{equation*}
F^{(2)}_{\boldsymbol{\alpha},\boldsymbol{\beta}}(q^d;x)=\sum_{n=0}^\infty\frac{(q^{d\alpha_1};q^d)_n\cdots(q^{d\alpha_v};q^d)_n}{(q^{d\beta_1};q^d)_n\cdots(q^{d\beta_w};q^d)_n}\cdot (1-q^d)^{(w-v)n}x^n \,.
\end{equation*}
Setting $\boldsymbol{r}:=((d\alpha_1,d),\ldots,(d\alpha_v,d))$ and $\boldsymbol{t}:=((d\beta_1,d),\ldots,(d\beta_w,d))$, we obtain that 
\begin{align*}
F^{(1)}_{\boldsymbol{\alpha},\boldsymbol{\beta}}(q^d;x) &= \sum_{n=0}^\infty  \mathcal{Q}_{\mathbf{r},\mathbf{t}}(q; n)(1-q^d)^{(w-v)n}q^{d(w-v)\binom{n}{2}}x^n\,, \\
F^{(2)}_{\boldsymbol{\alpha},\boldsymbol{\beta}}(q^d;x) &= \sum_{n=0}^\infty  \mathcal{Q}_{\mathbf{r},\mathbf{t}}(q; n)(1-q^d)^{(w-v)n} x^n \,,\\
F_{\boldsymbol{r},\boldsymbol{t}}(q;x)&= \sum_{n=0}^\infty  \mathcal{Q}_{\mathbf{r},\mathbf{t}}(q; n)  \left(\frac{(1-q^{u_1})\cdots(1-q^{u_w})}{(1-q^{s_1})\cdots(1-q^{s_v})}\right)^n  x^n\,.
\end{align*}
Note that for $q$-integrality, we can omit factors of the form $h(q)^n$ with $h(q)\in\mathbb Q(q)$ such as 
$$
(1-q^d)^{(w-v)n} \quad \text{ and }\quad   \left(\frac{(1-q^{u_1})\cdots(1-q^{u_w})}{(1-q^{s_1})\cdots(1-q^{s_v})}\right)^n   \,.
$$  
It follows  that 
$F^{(1)}_{\boldsymbol{\alpha},\boldsymbol{\beta}}(q;x)$ is $q^{1/d}$-integral if and only if $\mathcal{Q}_{\mathbf{r},\mathbf{t}}(q; n)$ is $q$-integral and 
$$
v_q\left( q^{d(w-v)\binom{n}{2}}\right)\geq an\, \quad\quad \forall n\geq 0\,,
$$
for some integer $a$, that is 
$$
F^{(1)}_{\boldsymbol{\alpha},\boldsymbol{\beta}}(q;x) \mbox{ is $q^{1/d}$-integral} \iff (\mathcal{Q}_{\mathbf{r},\mathbf{t}}(q; n))_{n\geq 0}  \mbox{ is $q$-integral and $w\geq v$.}
$$
We also deduce that 
$$F^{(2)}_{\boldsymbol{\alpha},\boldsymbol{\beta}}(q;x) \mbox{ is $q^{1/d}$-integral} \iff  F_{\boldsymbol{r},\boldsymbol{t}}(q;x) \mbox{ is $q$-integral}  \iff 
(\mathcal{Q}_{\mathbf{r},\mathbf{t}}(q; n))_{n\geq 0} \mbox{ is $q$-integral.}
$$

\subsection{Irreducible factors of $q$-Pochhammer symbols and $q$-integrality of $q$-hypergeometric sequences}\label{subsec: pre}

Throughout  this paper, we work only with ratios of products of terms of the form $(q^r;q^s)_n$ and $(1-q^s)$, where $r$ and $s$ are integers, $s\not=0$, and 
$n$ is an integer. 

Let us first recall that, for every positive integer $a$, we have  
\begin{equation}\label{eq: qa}
1-q^{a}= - \prod_{b\mid a} \phi_b(q) 
\quad 
\text{ and } 
\quad
1-q^{-a}=-q^{-a}(1-q^a) = q^{-a} \prod_{b\mid a} \phi_b(q) \,.
\end{equation}
Let $n\in \mathbb Z$. 
It follows that any ratio of products of terms of the form $(q^r;q^s)_n$ and $(1-q^s)$, where $r$ and $s$ are integers and $s\not=0$, has a unique decomposition of the form
\begin{equation}\label{eq: decomposition}
\pm q^{v_{q,n}}  \prod_{b=1}^{\infty} \phi_b(q)^{v_{b,n}} \,,
\end{equation}
where $v_{q,n},v_{1,n},\ldots$ are integers and $v_{b,n}=0$ for all but finitely many positive integers $b$. 
The integer $v_{q,n}$ is the $q$-valuation of this ratio and, for every $b\geq 1$, the integer $v_{b,n}$ is its $\phi_b$-valuation. 

\begin{rem}\label{rem: qq}\emph{
A term of the form \eqref{eq: decomposition} belongs to $\mathbb Z[q]$ if and only if the integers $v_{q,n},v_{1,n},v_{2,n},\ldots$ are all non-negative. When only the integers 
$v_{1,n},v_{2,n},\ldots$ are non-negative, then it belongs to $\mathbb Z[q^{-1},q]$.  }
\end{rem}

\subsubsection{The $q$-valuation of $q$-Pochhammer symbols}\label{subsec: qval}

Let $n$ be a positive integer and $r$ and $s$ be two integers, $s\not=0$. Let us assume that $(q^r;q^s)_n$ is well-defined and non-zero. 
We let $\mathcal N:=\{i\in\{0,\ldots,n-1\} : r+is <0\}$. Then we have
$$
v_q((q^r;q^s)_n) = \sum_{i\in\mathcal N} (r+is) \,.
$$
We deduce the following results. 
\begin{itemize}
\item[{\rm (i)}] When $r$ and $s$ are non-negative, then $v_q((q^r;q^s)_n)=0$. 

\item[{\rm (ii)}]  When $r$ is negative and $s$ positive, then the sequence $(v_q((q^r;q^s)_n))_{n\geq 0}$ remains bounded. 

\item[{\rm (iii)}]  When $s$ is negative, then 
\begin{equation}\label{eq: vqp}
v_q((q^r;q^s)_n)\underset{n\to+\infty}{\sim} s{n\choose 2}\,.
\end{equation} 
\end{itemize}

Now, let $n$ be a negative integer. We can derive similar results from the expression 
$$
(q^r;q^s)_n=\frac{1}{(q^{r-s};q^{-s})_{-n} }\,\cdot
$$
In particular, we get that $(v_q((q^r;q^s)_n))_{n\leq 0}$ remains bounded if $s$ is negative, and 
\begin{equation}\label{eq: vqpbis}
v_q((q^r;q^s)_n)\underset{n\rightarrow -\infty}{\sim} s\binom{-n}{2}\,
\end{equation} 
if $s$ is positive. 

\subsubsection{Asymptotics for cyclotomic and $q$-valuations of $q$-hypergeometric terms}
Let us consider the  $q$-hypergeometric sequence 
$$
Q_{\mathbf{r},\mathbf{t}}(q;n)=\frac{(q^{r_1};q^{s_1})_n\cdots(q^{r_v};q^{s_v})_n}{(q^{t_1};q^{u_1})_n\cdots(q^{t_w};q^{u_w})_n} \quad\quad n\geq 0\,,
$$
which we assume to be well-defined and not eventually zero. 
We first infer from  \eqref{eq: qa} that 
\begin{equation}\label{eq: asymp1}
v_{\phi_b}(Q_{\mathbf{r},\mathbf{t}}(q;n))=O(n)\,,
\end{equation}
for every positive integer $b$. 
Let $\mathcal N_1:=\{i \in\{1,\ldots,v\} : s_i<0\}$, $\mathcal N_2:=\{j \in\{1,\ldots,w\} : u_j<0\}$, and $s=\sum_{i\in\mathcal N_1} s_i - \sum_{j\in\mathcal N_2} u_j$. 
Using (i)--(iii) above, we deduce that 
\begin{equation}\label{eq: asymp2}
v_q(Q_{\mathbf{r},\mathbf{t}}(q;n))=  s{n\choose 2} + O(n)\,.
\end{equation}
It follows from \eqref{eq: decomposition}, Remark \ref{rem: qq}, and Equalities \eqref{eq: asymp1} and \eqref{eq: asymp2},  that 
\begin{equation}\label{eq: iff1}
(Q_{\mathbf{r},\mathbf{t}}(q;n))_{n\geq 0} \mbox{ is $q$-integral} \iff s\geq 0 \mbox { and } v_{\phi_b}(Q_{\mathbf{r},\mathbf{t}}(q;n))\geq 0\quad  \forall b\gg 1 
\end{equation}
and 
\begin{equation}\label{eq: iff2}
\exists C(q)\in\mathbb Z[q]\setminus\{0\} \mid \forall n \geq 0\,, \;C(q)^nQ_{\mathbf{r},\mathbf{t}}(q;n)\in\mathbb Z[q^{-1},q] 
\iff v_{\phi_b}(Q_{\mathbf{r},\mathbf{t}}(q;n))\geq 0   \quad  \forall b\gg 1 \,.
\end{equation}
The discussion of Section \ref{subsec: qval} also shows how to derive similar results for $q$-hypergeometric sequences of the form $(Q_{\mathbf{r},\mathbf{t}}(q;n))_{n\leq 0}$.

\section{The cyclotomic valuation of basic hypergeometric terms}\label{sec: val}

In this section, we introduce some generalizations of Dwork maps and Landau functions. They provide suitable tools to respectively 
compute the $\phi_b$-valuation of the $q$-Pochhammer symbol $(q^r;q^s)_n$ and of $q$-hypergeometric terms.   
Our approach takes its source in the works of Dwork \cite{Dwork},  Katz \cite{Katz}, and Christol \cite{Christol}. 
Precise formulas and properties for the $p$-adic valuation of Pochhammer symbols $(r/s)_n$ were given by Delaygue, Rivoal, and Roques \cite{DRR}  
in order to prove the integrality of coefficients of some mirror maps.  In this section, we  generalize those formulas, yielding finer results in analogy with 
\eqref{eq:LegendrePhib}. We also show that our results extend naturally to negative arguments $n$, and we derive new formulas that could be used to 
simplify  the proofs in \cite[Chapter 5]{DRR} considerably.

\subsection{A generalization of Dwork maps}\label{sec: dwork}

We first extend the definition of the Dwork map $D_p$, replacing the prime number $p$ by an arbitrary positive integer $b$. 
 
For every rational number $\alpha$, we let $d(\alpha)$ denote the exact positive denominator of $\alpha$, that is 
$$
d(\alpha):= \min \{d\in \mathbb N : \alpha = a/d, a\in \mathbb Z\} \,.
$$ 
Hence $d(\alpha)=1$ if and only if $\alpha$ is an integer. We also let $n(\alpha)$ denote the numerator of $\alpha$, that is the unique integer 
such that $\alpha=n(\alpha)/d(\alpha)$. 
For every positive integer $b$, we consider the multiplicative set $S_b:=\{k\in\mathbb{Z}:\,\gcd(k,b)=1\}$. We let $S_b^{-1}\mathbb{Z}\subset \mathbb Q$ denote the localization of $\mathbb Z$ by $S_b$, that is the ring formed by the rational numbers 
$\alpha$ such that $d(\alpha)$ belongs to $S_b$. 

\begin{propdef}\label{propdef: db} 
Let $b$ be a positive integer and $\alpha$ be in $S_b^{-1}\mathbb{Z}$. There is a unique element $D_b(\alpha)$ of $S_b^{-1}\mathbb{Z}$ such that
\begin{equation}\label{db}
bD_b(\alpha)-\alpha\in\{0,\dots,b-1\} \,.
\end{equation}
Furthermore, the formula 
\begin{equation}\label{eq: Db formula}
D_b(\alpha)=a\alpha+\left\lfloor \frac{\alpha-1}{b}-a\alpha\right\rfloor +1
\end{equation}
holds true for every integer $a$ satisfying $ab\equiv 1\mod d(\alpha)$. 
\end{propdef}

\begin{rem}\label{rem: dmap}\emph{
Note that the map $D_b$ is only defined from $S_b^{-1}\mathbb{Z}$ into itself. 
When $b=1$, $S_b^{-1}\mathbb{Z}=\mathbb{Q}$ and $D_1$ is just the identity map of $\mathbb{Q}$. 
In fact, not only $D_b(\alpha)\in S_b^{-1}\mathbb{Z}$, but, more precisely, Equation \eqref{eq: Db formula} shows that $D_b(\alpha)\in\frac{1}{d(\alpha)}\mathbb Z$. }
\end{rem}

\begin{proof}
Let us first assume by contradiction that $D_b(\alpha)$ is not unique, and let $\theta_1>\theta_2$ be two distinct elements of $S_b^{-1}\mathbb{Z}$ satisfying Equation \eqref{db}. 
It would yield $b\geq 2$ and $b(\theta_1-\theta_2)\in\{1,\dots,b-1\}$. Therefore we would have $\theta_1-\theta_2\notin S^{-1}_b\mathbb{Z}$, which would provide a contradiction 
since $S^{-1}_b\mathbb{Z}$ is a ring. Hence $D_b(\alpha)$ is unique.
\medskip

Now we prove the existence of $D_b(\alpha)$ while establishing \eqref{eq: Db formula}. Since, by assumption, $\alpha$ belongs to $S_b^{-1}\mathbb{Z}$, we have 
$\gcd(d(\alpha),b)=1$, and integers $a$ such that $ab\equiv 1\mod d(\alpha)$ do exist. Let $a$ be such an integer and  set
$$
\theta := a\alpha+\left\lfloor \frac{\alpha-1}{b}-a\alpha\right\rfloor +1 \,.
$$
Observe that $\theta\in S_b^{-1}\mathbb{Z}$. Since $ba\equiv 1\mod d(\alpha)$,  $ba\alpha-\alpha$ is an integer and $b\theta-\alpha$ belongs to $\mathbb{Z}$. Furthermore, we have
$$
\frac{\alpha-1}{b}<\theta\leq\frac{\alpha-1}{b}+1 \,,
$$
which yields
$$
-1<b\theta-\alpha\leq b-1\,.
$$
Hence $D_b(\alpha)=\theta$, as expected. 
\end{proof}

Following Christol \cite{Christol}, we introduce some notation which allows us to simplify the expression of $D_b(\alpha)$ when $b$ is large enough. 
For every real number $x$, we let $\{x\}$ denote its fractional part and we set
$$
\langle x\rangle:=\left\{\begin{array}{cl}
\{x\} &\textup{if $x\notin\mathbb{Z}$}\,,\\
1 & \textup{otherwise}.
\end{array} \right.
$$
Hence $\langle x\rangle= 1-\{1-x\}$. For every rational number $\alpha$, we also define
$$
\mathfrak{n}_\alpha:=\left\{\begin{array}{cl}
n(\alpha) & \textup{if $\alpha\geq 0$} \,,\\
|n(\alpha)|+1 & \textup{otherwise}.
\end{array}\right.
$$

\begin{prop}\label{prop Db}
Let $b$ be a positive integer and $\alpha$ be in $S_b^{-1}\mathbb{Z}$. Let $a$ be an integer satisfying $ab\equiv 1\mod d(\alpha)$. Then we have the formula
$$
D_b(\alpha)=\langle a\alpha\rangle-\left\lfloor\langle a\alpha\rangle-\frac{\alpha}{b}\right\rfloor\,.
$$
Furthermore, if $b\geq \mathfrak{n}_\alpha$, then 
$$
D_b(\alpha)=\left\{\begin{array}{cl}
\langle a\alpha\rangle & \textup{if $\alpha\notin\mathbb{Z}_{\leq 0}$\,,}\\
0 & \textup{otherwise}.
\end{array}\right.
$$ 
\end{prop}

It follows that, for a fixed rational number $\alpha$, and for all integers $b\geq \mathfrak{n}_\alpha$ coprime to $d(\alpha)$, $D_b(\alpha)$ only depends 
on the residue class of $b$ modulo $d(\alpha)$. 

\begin{rem}\emph{
When $b$ is  prime and $\alpha\notin\mathbb{Z}_{\leq 0}$, Lemma 23 in \cite{DRR} shows that $D_b(\alpha)=\langle a\alpha\rangle$ for $b\geq d(\alpha)(|\lfloor 1-\alpha\rfloor|+\langle\alpha\rangle)$. The condition $b\geq\mathfrak{n}_\alpha$ slightly improves on this bound.  When $\alpha>0$, it makes no difference because $\mathfrak{n}_\alpha=n(\alpha)$ which can be written $d(\alpha)\alpha=d(\alpha)(-\lfloor 1-\alpha\rfloor+\langle\alpha\rangle)$, with $\lfloor 1-\alpha\rfloor\leq 0$. But when $\alpha<0$, we have $\mathfrak{n}_\alpha=|n(\alpha)|+1$ which may improve on the previous bound. For example, even for $\alpha=-1/2$, one finds that $d(\alpha)(|\lfloor 1-\alpha\rfloor|+\langle\alpha\rangle)=3$ while $\mathfrak{n}_\alpha=2$.}
\end{rem}

\begin{proof}
Since, by assumption, $a$ does not divide $d(\alpha)$, there is an integer $k$ such that $\langle a\alpha\rangle=k/d(\alpha)$ and $k\equiv a n(\alpha)\mod d(\alpha)$. 
Hence $bk\equiv n(\alpha)\mod d(\alpha)$ and $b\langle a\alpha\rangle-\alpha$ is an integer. It follows that
$$
b\left(\langle a\alpha\rangle-\left\lfloor\langle a\alpha\rangle-\frac{\alpha}{b}\right\rfloor\right)-\alpha\in\mathbb{Z}\,.
$$
Furthermore, we have
$$
\langle a\alpha\rangle-\frac{\alpha}{b}-1<\left\lfloor\langle a\alpha\rangle-\frac{\alpha}{b}\right\rfloor\leq\langle a\alpha\rangle-\frac{\alpha}{b}\,,
$$
so that
$$
0\leq b\left(\langle a\alpha\rangle-\left\lfloor\langle a\alpha\rangle-\frac{\alpha}{b}\right\rfloor\right)-\alpha<b\,.
$$
This proves the expected formula for $D_b(\alpha)$ by uniqueness.

Now, let us assume that $b\geq\mathfrak{n}_\alpha$. Then we have $|\alpha/b|\leq 1/d(\alpha)$ and (even if $\alpha$ is an integer) 
$$
\frac{1}{d(\alpha)}\leq \langle a\alpha\rangle\leq 1\,.
$$
If $\alpha$ is positive, then it follows that 
\begin{equation}\label{eq:GoalProp}
\left\lfloor\langle a\alpha\rangle-\frac{\alpha}{b}\right\rfloor=0\,,
\end{equation}
that is $D_b(\alpha)=\langle a\alpha\rangle$. If $\alpha=0$, then $D_b(\alpha)=0$. If $\alpha$ is negative, then $\mathfrak{n}_\alpha=|n(\alpha)|+1$ 
and we obtain that $|\alpha/b|<1/d(\alpha)$. Hence, either $\alpha$ is an integer and
$$
D_b(\alpha)=1-\left\lfloor 1-\frac{\alpha}{b}\right\rfloor=0\,,
$$ 
or we have
$$
\frac{1}{d(\alpha)}\leq \langle a\alpha\rangle\leq \frac{d(\alpha)-1}{d(\alpha)} 
$$
and $D_b(\alpha)=\langle a\alpha\rangle$. In all cases, we obtain the expected result. 
\end{proof}

We end this section with a simple rule about composition of Dwork maps.

\begin{prop}\label{prop:Db compo}
Let $b$ and $c$ be two positive integers, and let $\alpha$ be in $S_{bc}^{-1}\mathbb{Z}$. Then we have 
$$
D_b(D_c(\alpha))=D_{bc}(\alpha)\,.
$$
In particular, we have $D_b^n=D_{b^n}$, and if $b^n\geq\mathfrak{n}_\alpha$ is congruent to $1$ modulo $d(\alpha)$ and $\alpha\not\in\mathbb{Z}_{\leq 0}$, then we have $D_b^n(\alpha)=\langle\alpha\rangle$.
\end{prop}

\begin{proof}
We have
$$
bcD_b(D_c(\alpha))-\alpha=c\big(bD_b(D_c(\alpha))-D_c(\alpha)\big)+cD_c(\alpha)-\alpha \,,
$$
which belongs to $\{0,\dots,bc-1\}$. Hence $D_b(D_c(\alpha))=D_{bc}(\alpha)$ by uniqueness. By induction, we get that $D_b^n=D_{b^n}$. By Proposition \ref{prop Db}, 
if $\alpha\notin\mathbb{Z}_{\leq 0}$ and $b^n\geq\mathfrak{n}_\alpha$ is congruent to $1$ modulo $d(\alpha)$, then $D_b^n(\alpha)=D_{b^n}(\alpha)=\langle \alpha\rangle$. Indeed, since $b^n\equiv 1\mod d(\alpha)$, we can choose $a=1$.
\end{proof}

\subsection{The cyclotomic valuation of $q$-Pochhammer symbols}\label{sec: valpoc}

In this section, we rephrase Theorem \ref{thm: valuationq-pocch} as Proposition \ref{prop: valuationq-pocch} and then we prove the latter.

\begin{defn}\label{def: delta}
Let $r$, $s$, and $b$ be integers with $s\neq 0$ and $b\geq 1$. Set $\alpha:=r/s$, $c:=\gcd(r,s,b)$, $b':=b/c$, and $s':=s/c$. 
If $\gcd(s',b')=1$, then $D_{b'}\left(\alpha\right)$ is well-defined and we set 
\begin{equation}\label{eq: D 0 1}
\gamma:=D_{b'}\left(\alpha\right)+\frac{\lfloor 1-\alpha\rfloor}{b'}\,\cdot
\end{equation}
We  define the (upper semi-continuous) step function $\delta_b(r,s,\cdot):\mathbb R \to \mathbb R$ by:  
$$
\delta_b(r,s,x):=\left\{\begin{array}{cl}
\left\lfloor cx-\gamma\right\rfloor+1 & \textup{if $\gcd(s',b')=1$\,,}\\
0 & \textup{otherwise.}
\end{array}\right.
$$
\end{defn}

\begin{lem}\label{lem: gamma}
The real number $\gamma$ defined in \eqref{eq: D 0 1} belongs to $(0,1]$.
\end{lem}

\begin{proof}
By definition, $b'D_{b'}(\alpha)-\alpha$ belongs to $\{0,\dots,b'-1\}$ and 
$\alpha=\langle\alpha\rangle-\lfloor 1-\alpha\rfloor$, where $\langle\alpha\rangle$ belongs to $(0,1]$. Thus, we have
$$
0<\frac{\langle\alpha\rangle}{b'}\leq D_{b'}(\alpha)+\frac{\lfloor 1-\alpha \rfloor}{b'}\leq \frac{b'-1+\langle\alpha\rangle}{b'}\leq 1\,,
$$
as expected.
\end{proof}

\begin{prop}\label{prop: valuationq-pocch}
Let $r$, $s$, and $b$ be integers such that $s\neq 0$ and $b\geq 1$. Let $n$ be an integer such that $(q^r;q^s)_n$ is well-defined and non-zero. Then we have
$$ 
v_{\phi_b}((q^r;q^s)_n)=\delta_b(r,s,n/b)\,.
$$ 
\end{prop}

It follows that when $b$ divides both $r$ and $s$, then $c=b$, $b'=1$ and $\delta_b(r,s,n/b)=n$, as expected since $\phi_b(q)$ divides each factor $1-q^{r+is}$.  
In particular, this is the case when $b=1$.

\medskip

In order to prove Proposition \ref{prop: valuationq-pocch} for negative $n$, we need the following lemma. It is also used in the proof of our criterion for the 
$q$-integrality of $q$-hypergeometric sequences.

\begin{lem}\label{lem:NegativeVal}
Let $r$, $s$, and $n$ be integers with $s\neq 0$, and let $b$ be a positive integer. Then we have
\begin{equation}\label{eq:conclude}
\delta_b(r,s,-n/b)=-\delta_b(r-s,-s,n/b)\,.
\end{equation}
\end{lem}

\begin{proof}
We set $c:=\gcd(r,s,b)$ and write $b=cb'$ and $s=cs'$. Both sides of Equation \eqref{eq:conclude} are $0$ when $\gcd(s',b')\neq 1$, so we can assume 
that $s'$ and $b'$ are coprime. Set $\alpha:=r/s$ so that $1-\alpha = (r-s)/(-s)$. We have 
$$
\delta_b(r-s,-s,x)=\left\lfloor cx-D_{b'}(1-\alpha)-\frac{\lfloor \alpha\rfloor}{b'}\right\rfloor+1\,.
$$
Since $b'D_{b'}(\alpha)-\alpha$ belongs to $\{0,\dots,b'-1\}$, we have
$$
b'(1-D_{b'}(\alpha))-(1-\alpha)=b'-1-(b'D_{b'}(\alpha)-\alpha)\in\{0,\dots,b'-1\}\,,
$$
so that $D_{b'}(1-\alpha)=1-D_{b'}(\alpha)$. It follows that
\begin{equation}\label{eq:ref floor}
\delta_b(r-s,-s,n/b)=\left\lfloor \frac{n}{b'}+D_{b'}(\alpha)-\frac{\lfloor \alpha\rfloor}{b'}\right\rfloor\,.
\end{equation}
If $x\in\mathbb{R}$, we have $x\in\mathbb{Z}$ or $\lfloor -x\rfloor = -\lfloor x\rfloor-1$, which also yields $\alpha\in\mathbb{Z}$ or $\lfloor 1-\alpha\rfloor=-\lfloor\alpha\rfloor$.
\medskip

Let us first consider the case where $\alpha\notin\mathbb{Z}$. Then $\lfloor\alpha\rfloor=-\lfloor 1-\alpha\rfloor$ and the right hand-side of \eqref{eq:ref floor} becomes 
\begin{equation}\label{eq:ref floor 2}
\left\lfloor \frac{n}{b'}+D_{b'}(\alpha)+\frac{\lfloor 1-\alpha\rfloor}{b'}\right\rfloor\,.
\end{equation}
We have $n+b'D_{b'}(\alpha)-\alpha\in\mathbb{Z}$, but $\alpha\notin\mathbb{Z}$. 
Hence $n+b'D_{b'}(\alpha)+\lfloor 1-\alpha\rfloor$ is not an integer and \eqref{eq:ref floor 2} is equal to
$$
-\left\lfloor -\frac{n}{b'}-D_{b'}(\alpha)-\frac{\lfloor 1-\alpha\rfloor}{b'}\right\rfloor-1=-\delta_b(r,s,-n/b)\,,
$$
as expected.
\medskip

It remains to consider the case where $\alpha\in\mathbb Z$. Set $k:=-\delta_b(r,s,-n/b)\in\mathbb{Z}$. We have
$$
\left\lfloor -\frac{n}{b'}-D_{b'}(\alpha)-\frac{\lfloor 1-\alpha\rfloor}{b'}\right\rfloor=-k-1 \,,
$$
which yields the equivalences
\begin{align*}
-k-1\leq -\frac{n}{b'}-D_{b'}(\alpha)-\frac{\lfloor 1-\alpha\rfloor}{b'}< -k&\iff k < \frac{n}{b'}+D_{b'}(\alpha)+\frac{1-\alpha}{b'}\leq k+1\\
&\iff k-\frac{1}{b'} < \frac{n}{b'}+D_{b'}(\alpha)-\frac{\alpha}{b'}\leq k+1-\frac{1}{b'}\,\cdot\\
\end{align*}
Even if $b'=1$, we obtain that 
$$
\left\lfloor \frac{n}{b'}+D_{b'}(\alpha)-\frac{\alpha}{b'}\right\rfloor = k \,.
$$
Combined with \eqref{eq:ref floor}, this yields \eqref{eq:conclude} and ends the proof of the lemma.
\end{proof}

\begin{proof}[Proof of Proposition \ref{prop: valuationq-pocch}]
Set $r'=r/c$. We first consider the case $n\geq 0$. We assume that $(q^r;q^s)_n$ is non-zero, 
that is $\alpha\notin\mathbb{Z}_{\leq 0}$ or $n\leq -\alpha$. 

We observe that $b\mid (r+is)$ if and only if $b'\mid (r'+is')$. Since we have $\gcd(r',s',b')=1$, if $b'$ and $s'$ are not coprime, then $b\nmid (r+is)$ and 
 $v_{\phi_b}((q^r;q^s)_n)=0$. 
 
 We now assume  that $b'$ and $s'$ are coprime. 
We need to find, among the powers of $q$ in the product defining $(q^r;q^s)_n$, which are multiples of $b$. We have the following equivalences:
\begin{align*}
r'+is'\equiv0\mod b' &\iff i\equiv -\alpha\mod  b'S_{b'}^{-1}\mathbb Z\\
&\iff i\equiv b'D_{b'}(\alpha)-\alpha\mod b'\\
&\iff\exists k\in\mathbb{N},\,i=b'D_{b'}(\alpha)-\alpha+kb'\,,
\end{align*}
because $i\geq 0$ and $b'D_{b'}(\alpha)-\alpha$ belongs to $\{0,\dots,b'-1\}$. We aim to count how many such integers $i$ belong to $\{0,\dots,n-1\}$. 
Writing $n-1=v+mb'$, with $0\leq v\leq b'-1$, and setting $\eta:=b'D_{b'}(\alpha)-\alpha$, we find all the integers $\eta,\eta+b',\dots,\eta+(m-1)b'$.  
Therefore, we have at least $m$ such integers $i$. There is one more such integer if and only if $v\geq b'D_{b'}(\alpha)-\alpha$. 
Furthermore, we have 
\begin{align}
v\geq b'D_{b'}(\alpha)-\alpha &\iff v+1\geq b'D_{b'}(\alpha)+\lfloor 1-\alpha\rfloor \label{eq: equi1}\\
&\iff \frac{v+1}{b'}\geq D_{b'}(\alpha)+\frac{\lfloor 1-\alpha\rfloor}{b'}\,\cdot \label{eq: equi2}
\end{align}
Equivalence \eqref{eq: equi1} follows from the implication 
\begin{align*}
v+1\geq b'D_{b'}(\alpha)+\lfloor 1-\alpha\rfloor &\Rightarrow v+1-\langle\alpha\rangle \geq b'D_{b'}(\alpha)-\alpha \\
&\Rightarrow v \geq b'D_{b'}(\alpha)-\alpha\,,
\end{align*}
because $1-\langle\alpha\rangle$ belongs to $[0,1)$. By Lemma \ref{lem: gamma}, since both sides of Inequality \eqref{eq: equi2} belong to $(0,1]$, we obtain that 
\begin{align*}
v_{\phi_b}((q^r;q^s)_n)&=m+\left\lfloor \frac{v+1}{b'}-D_{b'}(\alpha)-\frac{\lfloor 1-\alpha\rfloor}{b'}\right\rfloor+1\\
&=\left\lfloor \frac{n}{b'}-D_{b'}(\alpha)-\frac{\lfloor 1-\alpha\rfloor}{b'}\right\rfloor+1 \,,
\end{align*}
as expected. 
\medskip

We now assume that $n<0$ and that $(q^r;q^s)_n$ is well-defined, that is $\alpha\notin\mathbb{Z}_{>0}$ or $n> -\alpha$. We have
$$
(q^r;q^s)_n=\frac{1}{(q^{r-s};q^{-s})_{-n}} \,\cdot
$$
Using the non-negative case, we get that $v_{\phi_b}((q^r;q^s)_n)=-\delta_{b}(r-s,-s,-n/b)$. By Lemma~\ref{lem:NegativeVal}, the latter is equal to $\delta_b(r,s,n/b)$. 
This ends the proof.
\end{proof}

\subsection{Extension to $q$-hypergeometric terms}\label{sec: Delta}

Let $$\mathbf r=((r_1,s_1),\ldots,(r_v,s_v)) \quad\mbox{ and }\quad \mathbf t=((t_1,u_1),\ldots,(t_w,u_w))$$ be two vectors with integer coordinates and such that $s_1,\ldots,s_v,u_1,\ldots,u_w$ are non-zero.  Set $\alpha_i:=r_i/s_i$ and $\beta_j:=t_j/u_j$. For every non-negative $n$, the ratio
$$
Q_{\mathbf{r},\mathbf{t}}(q;n)=\frac{(q^{r_1};q^{s_1})_n\cdots(q^{r_v};q^{s_v})_n}{(q^{t_1};q^{u_1})_n\cdots(q^{t_w};q^{u_w})_n}
$$
is well-defined if we have either $\beta_j\notin\mathbb{Z}_{\leq 0}$ or $n\leq -\beta_j$, for all $j\in\{1,\ldots,w\}$. 
According to \eqref{eq: hgneg}, this $q$-hypergeometric term admits the following  extension to negative $n$:
$$
Q_{\mathbf{r},\mathbf{t}}(q;n)=\frac{(q^{t_1-u_1};q^{-u_1})_{-n}\cdots(q^{t_w-u_w};q^{-u_w})_{-n}}{(q^{r_1-s_1};q^{-s_1})_{-n}\cdots(q^{r_v-s_v};q^{-s_v})_{-n}}\,\cdot
$$
The latter is well-defined if  we have either $\alpha_i\notin\mathbb{Z}_{>0}$ or $n>-\alpha_i$, for all $i\in\{1,\ldots,v\}$.

If $R(q)$ and $S(q)$ are non-zero elements in $\mathbb{Z}[q^{-1},q]$, we write $R(q)\sim S(q)$ when $R(q)/S(q)$ is  a unit of $\mathbb{Z}[q^{-1},q]$, that is when it 
is of the form $\epsilon q^m$ with $\epsilon\in \{-1,1\}$ and $m\in\mathbb Z$. 

We introduce now some step functions that generalize the Landau functions mentioned in the introduction.

\begin{defn}\label{def: Delta}
We continue with the notation of Section \ref{sec: valpoc}. For every integer $b$, we define the (upper semi-continuous) step function $\Delta^{\mathbf{r},\mathbf{t}}_b : \mathbb R \to \mathbb R$ by: 
$$
\Delta^{\mathbf{r},\mathbf{t}}_b(x):=\sum_{i=1}^v\delta_b(r_i,s_i,x)-\sum_{j=1}^w\delta_b(t_j,u_j,x) \,.
$$
\end{defn}

As a direct consequence of Proposition \ref{prop: valuationq-pocch}, we deduce the following result.

\begin{cor}\label{coro:valuation}
Let $n\in\mathbb Z$ be such that $Q_{\mathbf{r},\mathbf{t}}(q;n)$ is well-defined and non-zero. 
Then we have 
$$
v_{\phi_b}(Q_{\mathbf{r},\mathbf{t}}(q;n))=\Delta^{\mathbf{r},\mathbf{t}}_b(n/b)\,,
$$ 
that is
\begin{equation}\label{eq: polyQalphabeta}
Q_{\mathbf{r},\mathbf{t}}(q;n)\sim\prod_{b=1}^\infty\phi_b(q)^{\Delta^{\mathbf{r},\mathbf{t}}_b(n/b)} \,.
\end{equation}
\end{cor}

\begin{rem}\emph{Let $\boldsymbol{\alpha}:=(\alpha_1,\dots,\alpha_v)$ and $\boldsymbol{\beta}:=(\beta_1,\dots,\beta_w)$ be vectors of rational numbers. Let $d:=d_{\boldsymbol{\alpha},\boldsymbol{\beta}}$ be the least common multiple of the denominators of the rational numbers $\alpha_i$ and $\beta_j$. Let $n$ be an integer such that
\begin{equation}\label{defQtilde}
\widetilde{Q}_{\boldsymbol{\alpha},\boldsymbol{\beta}}(q;n):=\frac{(q^{\alpha_1};q)_n\cdots(q^{\alpha_v};q)_n}{(q^{\beta_1};q)_n\cdots(q^{\beta_w};q)_n}
\end{equation}
is well-defined and non-zero. Then $\widetilde{Q}_{\boldsymbol{\alpha},\boldsymbol{\beta}}(q;n)$ belongs to $\mathbb{Q}(q^{1/d})$. By Remark \ref{rem: iso}, 
 Corollary~\ref{coro:valuation} implies that 
\begin{equation}\label{eq: polytildeQalphabeta}
\widetilde{Q}_{\boldsymbol{\alpha},\boldsymbol{\beta}}(q;n)\sim\prod_{b=1}^\infty
\phi_b\left(q^{1/d}\right)^{\Delta_b^{\mathbf{r},\mathbf{t}}(n/b)} \,,
\end{equation}
where the equivalence relation $\sim$ has to be understood in $\mathbb{Z}[q^{-1/d},q^{1/d}]$, and where 
$$
\mathbf{r}=((d\alpha_1,d),\dots,(d\alpha_v,d))\quad\textup{and}\quad\mathbf{t}=((d\beta_1,d),\dots,(d\beta_w,d))\,. 
$$}
\end{rem}

\section{First criteria for $q$-integrality of basic hypergeometric sequences}\label{sec:CriterionIntegrality}

In this section, we provide a criterion for the 
$q$-integrality of the $q$-hypergeometric sequences in terms of the Landau functions $\Delta^{\mathbf{r},\mathbf{t}}_b$, as well as related results.

\subsection{A first criterion of $q$-integrality}\label{sec: first criteria}

Our first result reads as follows.

\begin{prop}\label{propo: N-integrality} We continue with the notation of the previous sections. 
Let us assume that $(Q_{\mathbf{r},\mathbf{t}}(q;n))_{n\geq 0}$ is a well-defined sequence. Then the two following assertions are equivalent.
\begin{itemize}
\item[$\mathrm{(i)}$] There exists $C(q)\in\mathbb{Z}[q]\setminus\{0\}$ such that, for every $n\geq 0$, $C(q)^nQ_{\mathbf{r},\mathbf{t}}(q;n)\in\mathbb{Z}[q^{-1},q]$.
\item[$\mathrm{(ii)}$] For all but finitely many positive integers $b$, $\Delta_b^{\mathbf{r},\mathbf{t}}$ is non-negative on $\mathbb{R}_{\geq 0}$.
\end{itemize}
\end{prop}

According to \eqref{eq: iff1}, we deduce from Proposition~\ref{propo: N-integrality} the following result.

\begin{cor}\label{coro:integrality}
Let us assume that $(Q_{\mathbf{r},\mathbf{t}}(q;n))_{n\geq 0}$ is a well-defined sequence.  
Let $\mathcal N_1:=\{i \in\{1,\ldots,v\} : s_i<0\}$, $\mathcal N_2:=\{j \in\{1,\ldots,w\} : u_j<0\}$, and $s= \sum_{i\in\mathcal N_1} s_i-\sum_{j\in\mathcal N_2} u_j $. 
Let us assume that $s\geq 0$. Then the two following assertions are equivalent.
\begin{itemize}
\item[$\mathrm{(i)}$] The sequence $(Q_{\mathbf{r},\mathbf{t}}(q;n))_{n\geq 0}$ is $q$-integral.
\item[$\mathrm{(ii)}$] For all but finitely many positive integers $b$, $\Delta_b^{\mathbf{r},\mathbf{t}}$ is non-negative on $\mathbb{R}_{\geq 0}$.
\end{itemize}
\end{cor}

Throughout this section, we fix  $\mathbf{r}$ and $\mathbf{t}$, and we write $\Delta_b$ as a shorthand for $\Delta_b^{\mathbf{r},\mathbf{t}}$. 
Before proving Proposition \ref{propo: N-integrality}, we need to establish the following lemma about the jumps of Landau step functions.

\begin{lem}\label{lem:DeltaJumps}
For every integers $k$ and $b\geq 1$, and every real number $x$, we have $$\Delta_b(x+k)=\Delta_b(x)+k\Delta_b(1)\,.$$ 
Furthermore, if $b$ is large enough, then the distance between any two distinct jumps of $\Delta_b$ is greater than or equal to $1/b$. 
\end{lem} 

\begin{rem}\label{rem:Delta01}\emph{
By Lemma \ref{lem:DeltaJumps}, $\Delta_b$ is non-negative on $\mathbb{R}_{\geq 0}$ if and only if $\Delta_b$ is non-negative on $[0,1]$. In addition, when $b$ is coprime to $d_{\mathbf{r},\mathbf{t}}$, then $\Delta_b(1)=v-w$ and Assertion $\mathrm{(ii)}$ of Proposition \ref{propo: N-integrality} implies that $v\geq w$.}
\end{rem}

\begin{proof}
Let us first give a useful expression for $\Delta_b$. For all $i$ and $j$, we recall that $\alpha_i=r_i/s_i$ and $\beta_j=t_j/u_j$. We also set $c_i:=\gcd(r_i,s_i,b)$, 
$d_j:=\gcd(t_j,u_j,b)$,  and
\begin{equation}\label{eq: vbwb}
V_b:=\{1\leq i\leq v\,:\,\gcd(s_i,b)=c_i\}\quad\textup{and}\quad W_b:=\{1\leq j\leq w\,:\,\gcd(u_j,b)=d_j\} \,.
\end{equation}
We observe that $i\in V_b$ if and only if $\delta_b(r_i,s_i,\cdot)$ is not the zero function, while $j\in W_b$ if and only if $\delta_b(t_j,u_j,\cdot)$ is not the zero function. 
It follows that 
\begin{equation}\label{eq:DeltaGeneral}
\Delta_b(x) = \sum_{i\in V_b}\left\lfloor c_ix-D_{b/c_i}(\alpha_i)-\frac{\lfloor 1-\alpha_i\rfloor}{b/c_i}\right\rfloor-\sum_{j\in W_b}\left\lfloor d_jx-D_{b/d_j}(\beta_j)-\frac{\lfloor 1-\beta_j\rfloor}{b/d_j}\right\rfloor+\#V_b-\#W_b\,.
\end{equation}
Since $b/c_i$ is coprime to $d(\alpha_i)$ and $b/d_j$ is coprime to $d(\beta_j)$, we infer from Lemma \ref{lem: gamma} that 
$$
D_{b/c_i}(\alpha_i)+\frac{\lfloor 1-\alpha_i\rfloor}{b/c_i}\in(0,1] \quad\textup{and}\quad D_{b/d_j}(\beta_i)+\frac{\lfloor 1-\beta_j\rfloor}{b/d_j}\in(0,1]  \,.
$$
By Equality \eqref{eq:DeltaGeneral}, we first deduce that $\Delta_b(1)=\sum_{i\in V_b}c_i-\sum_{j\in W_b}d_j$, and then that 
\begin{align*}
\Delta_b(x+k)&=\Delta_b(x)+\sum_{i\in V_b}c_ik-\sum_{j\in W_b}d_jk\\
&=\Delta_b(x)+k\Delta_b(1)\,,
\end{align*}
for every integer $k$. This proves the first part of the lemma.  

\medskip

By \eqref{eq:DeltaGeneral}, the jumps of the step function $\Delta_b$ have abscissa of the form
\begin{equation}\label{eq:jumps}
\gamma(r,s,k):=\frac{D_{b/c}(\alpha)+k}{c}+\frac{\lfloor 1-\alpha\rfloor}{b},
\end{equation}
where $(r,s)$ belongs to $\mathbf{r}$ or $\mathbf{t}$, $\alpha=r/s$, $c=\gcd(r,s,b)$ and $k\in\mathbb{Z}$. Let $\gamma_1:=\gamma(r_1,s_1,k_1)$ and 
$\gamma_2:=\gamma(r_2,s_2,k_2)$ be two distinct abscissa of jumps as in \eqref{eq:jumps}. For $i=1$ or $2$, set $\alpha_i:=r_i/s_i$, $c_i:=\gcd(r_i,s_i,b)$ and $b_i:=b/c_i$. 
If we have
$$
\frac{D_{b_1}(\alpha_1)+k_1}{c_1}=\frac{D_{b_2}(\alpha_2)+k_2}{c_2}\,,
$$
then $\lfloor 1-\alpha_1\rfloor\neq\lfloor 1-\alpha_2\rfloor$ and 
$$
|\gamma_1-\gamma_2|=\frac{|\lfloor 1-\alpha_1\rfloor-\lfloor 1-\alpha_2\rfloor|}{b}\geq\frac{1}{b}\,,
$$
as expected. Otherwise, we get that 
$$
\left|\frac{D_{b_1}(\alpha_1)+k_1}{c_1}-\frac{D_{b_2}(\alpha_2)+k_2}{c_2}\right|\geq\frac{1}{d_{\mathbf{r},\mathbf{t}}}\,\cdot
$$
Indeed, we infer from Remark \ref{rem: dmap} that $D_{b_i}(\alpha_i)\in \frac{c_i}{s_i}\mathbb Z$, which shows that 
$$
\frac{D_{b_i}(\alpha_i)+k_i}{c_i}\in \frac{1}{s_i}\mathbb{Z}\,.
$$
Hence, for 
$
b\geq 2d_{\mathbf{r},\mathbf{t}}\cdot\max\{|\lfloor 1-\alpha\rfloor-\lfloor 1-\beta\rfloor|+1\,:\,\textup{$\alpha$ and $\beta$ in $\boldsymbol{\alpha}$ or $\boldsymbol{\beta}$}\},
$ 
we have 
$$
|\gamma_1-\gamma_2|>\frac{1}{2d_{\mathbf{r},\mathbf{t}}}  \geq \frac{1}{b}\,, 
$$
as expected. This ends the proof.
\end{proof}

\begin{proof}[Proof of Proposition \ref{propo: N-integrality}]
We first infer from \eqref{eq: iff2} and \eqref{eq: polyQalphabeta} that Assertion $\mathrm{(ii)}$ implies Assertion $\mathrm{(i)}$. 
Now, we assume that Assertion $\mathrm{(i)}$ holds and we prove Assertion $\mathrm{(ii)}$. 
By  \eqref{eq: iff2} and \eqref{eq: polyQalphabeta}, there exists a positive integer $m$ such that, for every non-negative integer $n$ and every integer $b\geq m$, 
we have $\Delta_b(n/b)\geq 0$. By Lemma~\ref{lem:DeltaJumps}, we can assume that $m$ is such that, for $b\geq m$, the distance between any two distinct jumps of $\Delta_b$ is greater than or equal to $1/b$. It follows that $\Delta_b$ is non-negative on $\mathbb{R}_{\geq 0}$ for all $b\geq m$, as wanted. 
\end{proof}

\subsection{Related criteria for negative arguments}\label{subsec: negativearguments}

It is easy to deduce from Proposition \ref{propo: N-integrality} a criterion for the $q$-integrality of the sequence $(Q_{\mathbf{r},\mathbf{t}}(q;-n))_{n\geq 0}$. 
Indeed, for every integer $n$, we have $Q_{\mathbf{r},\mathbf{t}}(q;n)=Q_{\mathbf{t}',\mathbf{r}'}(q;-n)$ (assuming that both terms are well-defined), 
where $\mathbf{r}'$ and $\mathbf{t}'$ are respectively obtained from $\mathbf{r}$ and $\mathbf{t}$ by replacing each pair $(r,s)$ in $\mathbf{r}$ or $\mathbf{t}$ by $(r-s,-s)$. 
By Lemma \ref{lem:NegativeVal}, 
for every positive integer $b$, we have 
$$
\Delta_b^{\mathbf{r},\mathbf{t}}(n/b)=\Delta_b^{\mathbf{t}',\mathbf{r}'}(-n/b)\,. 
$$
Combining Lemma \ref{lem:DeltaJumps} and Proposition \ref{propo: N-integrality},  we obtain that the following two assertions are equivalent. 
\begin{quote}
\begin{itemize}
\item[$\mathrm{(i)}$] There exists $C(q)\in\mathbb{Z}[q]\setminus\{0\}$ such that, for every $n\in\mathbb{Z}_{\leq 0}$, $C(q)^nQ_{\mathbf{r},\mathbf{t}}(q;n)$ 
belongs to $\mathbb{Z}[q^{-1},q]$.
\item[$\mathrm{(ii)}$] For all but finitely many positive integers $b$, $\Delta_b^{\mathbf{r},\mathbf{t}}$ is non-negative on $\mathbb{R}_{\leq 0}$.
\end{itemize}
\end{quote}

A natural question is then to ask whether it is possible to find a non-zero rational fraction $C(q)$ in $\mathbb{Q}(q)$ such that 
$C(q)^nQ_{\mathbf{r},\mathbf{t}}(q;n)$ is a polynomial for positive and negative $n$ simultaneously. 
The main problem is that the numerator of $C(q)$ will bring new denominators for negative $n$ and \textit{vice versa}. 
It turns out that this problem can be overcome only in the special case where $Q_{\mathbf{r},\mathbf{t}}(q;n)\in\mathbb{Z}[q^{-1},q]$ for all integers $n$.

\begin{prop}
Let us assume that $(Q_{\mathbf{r},\mathbf{t}}(q;n))_{n\in\mathbb{Z}}$ is a well-defined family. Then the three following assertions are equivalent.
\begin{itemize}
\item[$\mathrm{(i)}$] There exists $C(q)\in\mathbb{Q}[q]\setminus\{0\}$ such that, for every $n\in\mathbb{Z}$, $C(q)^nQ_{\mathbf{r},\mathbf{t}}(q;n)\in\mathbb{Z}[q^{-1},q]$.
\item[$\mathrm{(ii)}$] For every $n\in\mathbb{Z}$, $Q_{\mathbf{r},\mathbf{t}}(q;n)\in\mathbb{Z}[q^{-1},q]$.
\item[$\mathrm{(iii)}$] For every $n\in\mathbb{N}$, $Q_{\mathbf{r},\mathbf{t}}(q;n)\in\mathbb{Z}[q^{-1},q]$ and all but finitely many positive integers $b$, 
$\Delta_b^{\mathbf{r},\mathbf{t}}$ is $1$-periodic. 
\end{itemize}
\end{prop}

\begin{proof}
Let us first prove that $\mathrm{(i)}$ implies $\mathrm{(iii)}$. If we assume $\mathrm{(i)}$, then, by the above criteria, for every large enough positive integer $b$, 
$\Delta_b$ is non-negative on $\mathbb{R}$. By Lemma \ref{lem:DeltaJumps}, we obtain that $\Delta_b(1)=0$ and that $\Delta_b$ is $1$-periodic. 
Even for small positive integers $b$, we have 
$$
\Delta_b(1)=\sum_{i\in V_b}c_i-\sum_{j\in W_b}d_j \,,
$$
where $V_b$, $W_b$, $c_i$ and $d_j$ are defined as in \eqref{eq: vbwb}.  The latter only depend on the congruence class of $b$ modulo $d_{\mathbf{r},\mathbf{t}}$. Hence $\Delta_b(1)=\Delta_{b+ld_{\mathbf{r},\mathbf{t}}}(1)$, while $\Delta_{b+ld_{\mathbf{r},\mathbf{t}}}(1)=0$ for $l$ large enough. It follows that $\Delta_b(1)=0$ and $\Delta_b$ is $1$-periodic for every positive integer $b$. 
In particular, if $\Delta_b(n/b)<0$ for some positive integers $n$ and $b$, then there exists a negative integer $m$ such that $\Delta_b(m/b)<0$. In this case, both the 
$\phi_b$-valuation of $Q_{\mathbf{r},\mathbf{t}}(q;n)$ and $Q_{\mathbf{r},\mathbf{t}}(q;m)$ are negative, which contradicts $\mathrm{(i)}$. 
It follows that, for every $n\in\mathbb{N}$, $Q_{\mathbf{r},\mathbf{t}}(q;n)\in\mathbb{Z}[q^{-1},q]$ and $\mathrm{(iii)}$ is proved.

Now, let us prove that $\mathrm{(iii)}$ implies $\mathrm{(ii)}$. If $\mathrm{(iii)}$ holds, then, reasoning as above, we get that $\Delta_b$ is $1$-periodic 
for all positive integers $b$. For all positive integers $n$ and $b$, we have $Q_{\mathbf{r},\mathbf{t}}(q;n)\in\mathbb{Z}[q^{-1},q]$, so that $\Delta_b(n/b)\geq 0$. 
By $1$-periodicity, for all integers $n$ and $b\geq 1$, we have $\Delta_b(n/b)\geq 0$, that is $Q_{\mathbf{r},\mathbf{t}}(q;n)\in\mathbb{Z}[q^{-1},q]$, as expected.

Obviously, $\mathrm{(ii})$ implies $\mathrm{(i)}$ by choosing $C(q)=1$, which ends the proof of the proposition.
\end{proof}

\subsection{Small digression on the step function $\Delta_b^{\mathbf{r},\mathbf{t}}$ }\label{sec: dig}

In this section, we use Proposition \ref{prop Db} to simplify the expression of $\Delta_b(x)$ when $b$ is large enough. To that end we introduce some additional notation. 
We continue with the notation introduced in \eqref{eq: vbwb} and we let $\mathfrak{n}_\alpha$  be defined as in Proposition \ref{prop Db}. 
We define 
$\mathfrak{a}_{\mathbf{r},\mathbf{t}}$ as the maximum of the numbers $\mathrm{gcd}(r_i,s_i)$ and $\mathrm{gcd}(t_j,u_j)$ for all $i$ and $j$. We set
$$
\mathfrak{n}_{\mathbf{r},\mathbf{t}}:=\mathrm{max}\{\mathfrak{n}_\alpha\,:\,\alpha\textup{ in $\boldsymbol{\alpha}$ or $\boldsymbol{\beta}$}\}\quad\textup{and}\quad\mathfrak{b}_{\mathbf{r},\mathbf{t}}:=\mathfrak{a}_{\mathbf{r},\mathbf{t}}\cdot\mathfrak{n}_{\mathbf{r},\mathbf{t}}\,.
$$
Let $b\geq \mathfrak{b}_{\mathbf{r},\mathbf{t}}$ be a fixed integer. For every $i\in V_b$ and $j\in W_b$, there exist positive integers $e_i$ and $f_j$ such that 
$$
be_i\equiv c_i\mod s_i\quad\textup{and}\quad bf_j\equiv d_j\mod u_j\,.
$$ 

Now, take for example $i\in V_b$. We have
$$
\frac{b}{c_i}\geq\frac{\mathfrak{b}_{\mathbf{r},\mathbf{t}}}{c_i}\geq \frac{\mathfrak{a}_{\mathbf{r},\mathbf{t}}}{c_i}\mathfrak{n}_{\mathbf{r},\mathbf{t}}\geq\mathfrak{n}_{\alpha_i}\,.
$$
So we can apply Proposition \ref{prop Db} to obtain that $D_{b/c_i}(\alpha_i)=\langle e_i\alpha_i\rangle$ if $\alpha_i\notin\mathbb{Z}_{\leq 0}$ and $0$ otherwise. 
Let us consider a slight modification of the function $\langle\cdot\rangle$ defined for every $x\in\mathbb{R}$ by 
$$
\langle x\rangle^\ast:=\left\{\begin{array}{cl}
\{x\} & \textup{if $x\notin\mathbb{Z}$}\,,\\
1 & \textup{if $x\in\mathbb{Z}_{> 0}$}\,,\\
0 & \textup{otherwise}.
\end{array}\right.
$$
For $i\in V_b$, if $c_i< s_i$, then $e_i$ is invertible modulo $s_i/c_i$ which is a denominator of $\alpha_i$. It follows that $e_i\alpha_i\in\mathbb{Z}_{\leq 0}$ if and only if $\alpha_i\in\mathbb{Z}_{\leq 0}$. Hence, we deduce from \eqref{eq:DeltaGeneral} that, for all $b\geq\mathfrak{b}_{\mathbf{r},\mathbf{t}}$ and all $x\in\mathbb{R}$, $\Delta_b(x)$ is equal to

\begin{equation}\label{eq:DeltaSimplifiedGeneral}
\sum_{i\in V_b}\left\lfloor c_ix-\langle e_i\alpha_i\rangle^\ast-\frac{\lfloor 1-\alpha_i\rfloor}{b/c_i}\right\rfloor-\sum_{j\in W_b}\left\lfloor d_jx-\langle f_j\beta_j\rangle^\ast-\frac{\lfloor 1-\beta_j\rfloor}{b/d_j}\right\rfloor+\#V_b-\#W_b\,.
\end{equation}

Let $d_{\mathbf{r},\mathbf{t}}$ be the least common multiple of the integers $s_1,\dots,s_v,u_1,\dots,u_w$. If in addition $b$ is coprime to $d_{\mathbf{r},\mathbf{t}}$, then all the  numbers $c_i$ and $d_j$ are equal to $1$. Let $a$ in $\{1,\dots,d_{\mathbf{r},\mathbf{t}}\}$ be such that $ab\equiv 1\mod d_{\mathbf{r},\mathbf{t}}$. Then, 
for all $i$ and $j$, we can take $e_i=f_j=a$, so that
$$
\Delta_b(x)=\sum_{i=1}^v\left\lfloor x-\langle a\alpha_i\rangle^\ast-\frac{\lfloor 1-\alpha_i\rfloor}{b}\right\rfloor-\sum_{j=1}^w\left\lfloor x-\langle a\beta_j\rangle^\ast-\frac{\lfloor 1-\beta_j\rfloor}{b}\right\rfloor+v-w\,.
$$
Moreover, if all the numbers $\alpha_i$ and $\beta_j$ belong to $(0,1]$, then we have
$$
\Delta_b(x)=\sum_{i=1}^v\lfloor x-\langle a\alpha_i\rangle\rfloor-\sum_{j=1}^w\lfloor x-\langle a\beta_j\rangle\rfloor+v-w\,,
$$
which only depends on the congruence class of $b$ modulo $d_{\mathbf{r},\mathbf{t}}$.

\section{Efficient criteria for $q$-integrality of basic hypergeometric sequences}\label{sec: effi}

To verify the second assertion in Proposition \ref{propo: N-integrality} and in Corollary \ref{coro:integrality}, 
we need in principle to perform infinitely many tests, checking the non-negativity of the step function $\Delta_b^{\mathbf{r},\mathbf{t}}$ on $\mathbb R_{\geq 0}$ 
for all sufficiently large integers $b$. This is not entirely satisfactory and the aim of Theorem~\ref{thm:N-integralityPolyn} is precisely to reduce the situation to 
a finite number of similar tests. In this section, we introduce the step functions $\Xi_{\mathbf{r},\mathbf{t}}(b,\cdot)$, $b\in\{1,\ldots,d_{\mathbf{r},\mathbf{t} }\}$. 
Then we prove Theorem \ref{thm:N-integralityPolyn}.

\subsection{A generalization of Christol step functions}\label{sec:fcChristol}

Following Christol \cite{Christol}, we define a total order $\preceq$ on $\mathbb{R}$ as follows. For all real numbers $x$ and $y$, we set
$$
x\preceq y\iff(\langle x\rangle<\langle y\rangle\textup{ or }(\langle x\rangle=\langle y\rangle\textup{ and }x\geq y))\,.
$$
We refer to it  as Christol order. 
Let  $\boldsymbol{\alpha}:=(\alpha_1,\ldots,\alpha_v)$ and $\boldsymbol{\beta}:=(\beta_1,\ldots,\beta_w)$ be two vectors 
of rational numbers, and  
$$
d_{\boldsymbol{\alpha},\boldsymbol{\beta}}:=\lcm(d(\alpha_1),\ldots,d(\alpha_v),d(\beta_1),\ldots,d(\beta_w))\,.
$$ 
For every integers $a\in \{1,\ldots,d_{\boldsymbol{\alpha},\boldsymbol{\beta}}\}$ coprime to $d_{\boldsymbol{\alpha},\boldsymbol{\beta}}$, Christol defined 
the step function $\xi_{\boldsymbol{\alpha},\boldsymbol{\beta}}(a,\cdot)$ from $\mathbb R$ to $\mathbb R$ by: 
\begin{equation}\label{eq: fchris}
\xi_{\boldsymbol{\alpha},\boldsymbol{\beta}}(a,x):=\#\{i \in \{1,\ldots,v\}\,:\,a\alpha_i\preceq x\} - \#\{j \in \{1,\ldots,w\} \,:\,a\beta_j\preceq x\}\,.
\end{equation}

We recall here our notation. Let $v$ and $w$ be positive integers, and for $i\in\{1,\dots,v\}$ and $j\in\{1,\dots,w\}$, 
let $(r_i,s_i)$ and $(t_j,u_j)$ be pairs of integers such that $s_iu_j\not=0$ for all $(i,j)$. Set $\alpha_i:=r_i/s_i$, $\beta_j:=t_j/u_j$, 
$\mathbf{r}:=((r_1,s_1),\dots,(r_v,s_v))$, $\mathbf{t}:=((t_1,u_1),\dots,(t_w,u_w))$, 
$\boldsymbol{\alpha}:=(\alpha_1,\dots,\alpha_v)$, $\boldsymbol{\beta}:=(\beta_1,\dots,\beta_w)$, and 
$d_{\mathbf{r},\mathbf{t}}:=\lcm(s_1,\dots,s_v,u_1,\dots,u_w)$.  

\medskip 

For every $b\in\{1,\dots,d_{\mathbf{r},\mathbf{t}}\}$, we define the step function $\Xi_{\mathbf{r},\mathbf{t}}(b,\cdot)$ as follows. 
For all $i\in\{1,\dots,v\}$ and $j\in\{1,\dots,w\}$, we set $c_i:=\gcd(r_i,s_i,b)$ and $d_j:=\gcd(t_j,u_j,b)$. We consider, as in \eqref{eq: vbwb}, the sets of indices
$$
V_b:=\{1\leq i\leq v\,:\,\gcd(s_i,b)=c_i\}\quad\textup{and}\quad W_b:=\{1\leq j\leq w\,:\,\gcd(u_j,b)=d_j\}\,. 
$$
As we already observed in Section \ref{sec: dig}, for every $i\in V_b$ and $j\in W_b$, there exist positive integers $e_i$ and $f_j$ such that 
$$
be_i\equiv c_i\mod s_i\quad\textup{and}\quad bf_j\equiv d_j\mod u_j\,.
$$
For all $i,j$, we choose such integers $e_i$ and $f_j$. We stress that the definition of $\Xi_{\mathbf{r},\mathbf{t}}(b,\cdot)$ (see Definition \ref{def: gchris}) 
does not depend on this choice. 
Let $\tilde{b}$ be the greatest divisor of $b$ coprime to $d_{\mathbf{r},\mathbf{t}}$ and let $a$ be the unique element of $\{1,\dots,d_{\mathbf{r},\mathbf{t}}\}$ satisfying 
$a\tilde{b}\equiv 1\mod d_{\mathbf{r},\mathbf{t}}$.

\begin{defn}\label{def: gchris}
For every integer $b$ in $\{1,\ldots,d_{\mathbf{r},\mathbf{t} }\}$, we define the step function 
$\Xi_{\mathbf{r},\mathbf{t}}(b,\cdot): \mathbb R\to\mathbb R$ by: 
\begin{align*}
\Xi_{\mathbf{r},\mathbf{t}}(b,x):=&\#\left\{(i,k)\in V_b\times \{0,\ldots,c_i-1\}\,:\, \frac{\langle e_i\alpha_i\rangle+k}{c_i}-\lfloor 1-a\alpha_i\rfloor\  \preceq x\right\}\\
&-
\#\left\{(j,\ell)\in W_b\times \{0,\ldots,d_j-1\}\,:\, \frac{\langle f_j\beta_j\rangle+\ell}{d_j}-\lfloor 1-a\beta_j\rfloor \preceq x\right\}\,.
\end{align*}
\end{defn}

\subsection{Comparison with the step functions $\xi_{\boldsymbol{\alpha},\boldsymbol{\beta}}(a,\cdot)$ and $\Delta_b^{\mathbf{r},\mathbf{t}}$}

The functions $\Xi_{\mathbf{r},\mathbf{t}}(b,\cdot)$ can be thought of as a generalization of the functions $\xi_{\boldsymbol{\alpha},\boldsymbol{\beta}}(a,\cdot)$ 
to composite numbers $b$. 
Indeed, if we assume that $b$ is coprime to $d_{\mathbf{r},\mathbf{t}}$ and that all the ratios $\alpha_i=r_i/s_i$ and $\beta_j=t_j/u_j$ 
belong to $\mathbb Q\setminus\mathbb Z_{\leq0}$, we claim that $\Xi_{\mathbf{r},\mathbf{t}}(b,\cdot)=\xi_{\boldsymbol{\alpha},\boldsymbol{\beta}}(a,\cdot)$ where $ab\equiv 1\mod d_{\mathbf{r},\mathbf{t}}$. 

Let us prove this claim. 
If $b$ is coprime to $d_{\mathbf{r},\mathbf{t}}$,  then $b=\tilde{b}$, all the numbers $c_i$ and $d_j$ are equal to $1$, $V_b=\{1,\ldots v\}$, $W_b=\{1,\ldots,w\}$.
Hence, for all $i$ and $j$, we can choose $e_i=f_j=a$. 
Moreover, for all $(i,k)\in V_b\times\{0,\ldots,c_i-1\}$, we have  $k=0$. We obtain that 
$$
\frac{\langle e_i\alpha_i\rangle+k}{c_i}-\lfloor 1-a\alpha_i\rfloor=\langle a\alpha_i\rangle-\lfloor 1-a\alpha_i\rfloor=a\alpha_i \,.
$$
Similarly, for all $(j,\ell)\in W_b\times\{0,\ldots,d_j-1\}$, we have
$$
\frac{\langle f_j\beta_j\rangle+\ell}{d_j}-\lfloor 1-a\beta_j\rfloor=a\beta_j\,.
$$  
By  \eqref{eq: fchris},  we get that 
\begin{align*}
\Xi_{\mathbf{r},\mathbf{t}}(b,x)&=  \#\left\{(i,k)\in \{1,\ldots,v\}\times \{0\} : a\alpha_i \preceq x\right\} -  \#\left\{(j,\ell)\in \{1,\ldots,w\}\times\{0\} : a\beta_j \preceq x\right\} \\
&= \#\left\{i\in\{1,\ldots,v\} : a\alpha_i \preceq x\right\} -  \#\left\{j\in\{1,\ldots,w\} : a\beta_j \preceq x\right\} \\
&= \xi_{\boldsymbol{\alpha},\boldsymbol{\beta}}(a,x)\,.
\end{align*}

\medskip

Let us now compare the step functions $\Xi_{\mathbf{r},\mathbf{t}}(b,\cdot)$ and $\Delta_b^{\mathbf{r},\mathbf{t}}$.  
Using Equality \eqref{eq:DeltaGeneral}, we can give a new expression for $\Delta_b^{\mathbf{r},\mathbf{t}}$ (restricted on $[0,1]$) 
which  is closer to the definition of the step function 
$\Xi_{\mathbf{r},\mathbf{t}}(b,\cdot)$. Indeed, for every positive integer $b$ and every real number $x$ in $[0,1]$, we get that 
\begin{align}\label{eq: deltanew}
\Delta_b^{\mathbf{r},\mathbf{t}}(x)= &\#\left\{(i,k)\in V_b\times \{0,\ldots,c_i-1\}\,:\, \frac{D_{b/c_i}(\alpha_i)+k}{c_i}+\frac{\lfloor 1-\alpha_i\rfloor}{b} \leq x\right\}\\
&-
\#\left\{(j,\ell)\in W_b\times \{0,\ldots,d_j-1\}\,:\, \frac{D_{b/d_j}(\beta_i)+\ell}{d_j}+\frac{\lfloor 1-\beta_j\rfloor}{b} \leq x\right\}\,.\nonumber
\end{align}

\subsection{Ordering of jumps}

The interest of the step-functions $\Xi_{\mathbf{r},\mathbf{t}}(b,\cdot)$ is that they keep track of all jumps configurations of the Landau 
functions $\Delta_\ell^{\mathbf{r},\mathbf{t}}$ for large $\ell$ congruent to $b$ modulo $d_{\mathbf{r},\mathbf{t}}$. More precisely, we have the following result.

\begin{lem}\label{lem: sauts}
For every $i\in\{1,2\}$, we let $r_i$ and $s_i$ be integers with $s_i\neq 0$ and such that $\alpha_i:=r_i/s_i\notin\mathbb{Z}_{\leq 0}$. 
Set $d:=\mathrm{lcm}(s_1,s_2)$ and let $b$ be an integer such that 
$$
b> \max\big(|r_1|,|r_2|,d\cdot|\lfloor 1-\alpha_1\rfloor-\lfloor 1-\alpha_2\rfloor|\big) \,.
$$ 
Set $c_i:=\gcd(r_i,s_i,b)$ and let us assume that there exists an integer $e_i$, $1\leq e_i\leq d$, such that $be_i\equiv c_i\mod s_i$. 
Let $k_i$ be an integer in $\{0,\dots, c_i-1\}$ and $a$ be a positive integer. Set 
$$
\gamma_i:=\frac{D_{b/c_i}(\alpha_i)+k_i}{c_i}+\frac{\lfloor 1-\alpha_i\rfloor}{b}\quad\textup{and}\quad\Gamma_i:=\frac{\langle e_i\alpha_i\rangle+k_i}{c_i}-\lfloor 1-a\alpha_i\rfloor\,.
$$
Then we have 
$$
\gamma_1\leq\gamma_2\iff\Gamma_1\preceq\Gamma_2\,.
$$
Furthermore, if $\Gamma_1=\Gamma_2$, then $\alpha_1=\alpha_2$.
\end{lem}

\begin{rem}\emph{
Contrary to what   the notation of Lemma \ref{lem: sauts} may suggest, we stress that the latter applies to compare the ordering of both the jumps with positive and negative amplitude of the step functions 
$\Xi_{\mathbf{r},\mathbf{t}}(b,\cdot)$ and $\Delta_b^{\mathbf{r},\mathbf{t}}$. }
\end{rem}

Even when $b\geq \mathfrak{b}_{\mathbf{r},\mathbf{t}}$, Formula \eqref{eq:DeltaSimplifiedGeneral} shows that the Landau functions $\Delta_b^{\mathbf{r},\mathbf{t}}$ 
depend in principle on $b$ and not only on the congruence class of $b$ modulo $d_{\mathbf{r},\mathbf{t}}$. In contrast, Lemma~\ref{lem: sauts} shows that for sufficiently large $b$, the $\leq$-ordering of the jumps of $\Delta_b^{\mathbf{r},\mathbf{t}}$ on $[0,1]$ is the same as the $\preceq$-ordering of that of 
$\Xi_{\mathbf{r},\mathbf{t}}(\underline{b},\cdot)$ on $\mathbb{R}$, where $\underline{b}$ is the unique representative in $\{1,\dots,d_{\mathbf{r},\mathbf{t}}\}$ of $b$ modulo $d_{\mathbf{r},\mathbf{t}}$. In particular, this ordering only depends on the congruence class of 
$b$ modulo $d_{\mathbf{r},\mathbf{t}}$.

Furthermore, Lemma \ref{lem: sauts} shows that if two jumps of $\Xi_{\mathbf{r},\mathbf{t}}(b,\cdot)$, respectively associated with the pairs $(r_1,s_1)$ and $(r_2,s_2)$, have the same abscissa, 
then we must have $r_1/s_1=r_2/s_2$. However, these pairs can still be distinct. Indeed, taking for example the pairs $(r_1,s_1)=(1,4)$ and 
$(r_2,s_2)=(3,12)$, and $b=9$, we find that $d=12$ and $\tilde{b}=1$, so that $a=1$,  $c_1=1$, $c_2=3$, $e_1=1$, and $e_2=3$. 
Hence taking $k_1=k_2=0$ yields
$$
\Gamma_1=\langle 1/4\rangle=\frac{1}{4}\quad\textup{and}\quad\Gamma_2=\frac{\langle 3/4\rangle}{3}=\frac{1}{4}\,\cdot
$$

\begin{proof}[Proof of Lemma \ref{lem: sauts}]
For $i\in\{1,2\}$, we set $b_i:=b/c_i$ and 
$$
\theta_i:=\frac{D_{b_i}(\alpha_i)+k_i}{c_i}\,\cdot
$$
Since $b>|r_i|$ and $c_i$ divides both $r_i$ and $s_i$, we have $b_i>|n(\alpha_i)|$ and hence $b_i\geq\mathfrak{n}_{\alpha_i}$. 
By Proposition~\ref{prop Db}, we have $D_{b_i}(\alpha_i)=\langle e_i\alpha_i\rangle$ for $\alpha_i\notin\mathbb{Z}_{\leq 0}$, so that 
\begin{equation}\label{eq: thetai}
\theta_i=\frac{\langle e_i\alpha_i\rangle+k_i}{c_i}\,\cdot
\end{equation}
Note that $\theta_i\in \frac{1}{d}\mathbb Z$. Indeed, $c_i$ divides $\mathrm{gcd}(r_i,s_i)$ so that $\alpha_i/c_i\in\frac{1}{s_i}\mathbb Z$, while $d$ is a multiple of $s_i$. 
Now, we show that
\begin{equation}\label{eq:FromThetatoAlpha}
\theta_1=\theta_2\Rightarrow\langle\alpha_1\rangle=\langle\alpha_2\rangle\, .
\end{equation}
Setting $s_i':=s_i/c_i$, for $i\in\{1,2\}$, we get that $b_ie_i\equiv 1\mod s_i'$ and $\alpha_i\in \frac{1}{s_i'}\mathbb Z$. We obtain that
$$
\langle b_i\langle e_i\alpha_i\rangle\rangle=\langle\alpha_i\rangle\,.
$$
Thereby, we obtain that
\begin{align*}
\theta_1=\theta_2&\Rightarrow b\theta_1=b\theta_2\\
&\Rightarrow b_1\langle e_1\alpha_1\rangle+b_1k_1=b_2\langle e_2\alpha_2\rangle+b_2k_2\\
&\Rightarrow \langle b_1\langle e_1\alpha_1\rangle\rangle=\langle b_2\langle e_2\alpha_2\rangle\rangle\\
&\Rightarrow\langle\alpha_1\rangle=\langle\alpha_2\rangle\,,
\end{align*}
which proves \eqref{eq:FromThetatoAlpha}. Furthermore, since $\alpha=\langle\alpha\rangle-\lfloor 1-\alpha\rfloor$, we have
\begin{align*}
\lfloor 1-a\alpha_i\rfloor&=\lfloor 1-a\langle\alpha_i\rangle+a\lfloor 1-\alpha_i\rfloor\rfloor\\
&=\lfloor 1-a\langle\alpha_i\rangle\rfloor+a\lfloor 1-\alpha_i\rfloor\,.
\end{align*}
If $\theta_1=\theta_2$, then we have $\langle\alpha_1\rangle=\langle\alpha_2\rangle$ and
\begin{equation}\label{eq: aalpha}
\lfloor 1-a\alpha_1\rfloor\geq\lfloor 1-a\alpha_2\rfloor\iff\lfloor 1-\alpha_1\rfloor\geq\lfloor 1-\alpha_2\rfloor\,,
\end{equation}
for $a$ is a positive integer. Since $\theta_i$ belongs to $\frac{1}{d}\mathbb{Z}$ and $b> d\cdot|\lfloor 1-\alpha_1\rfloor-\lfloor 1-\alpha_2\rfloor|$, we obtain the following equivalences:
\begin{align*}
\gamma_1\leq \gamma_2
& \iff\theta_1+\frac{\lfloor 1-\alpha_1\rfloor}{b}\leq\theta_2+\frac{\lfloor 1-\alpha_2\rfloor}{b}\\
& \iff\theta_1-\theta_2\leq\frac{\lfloor 1-\alpha_2\rfloor-\lfloor 1-\alpha_1\rfloor}{b}\\
& \iff \theta_1<\theta_2\textup{ or }(\theta_1=\theta_2\textup{ and }\lfloor 1-\alpha_1\rfloor\leq\lfloor 1-\alpha_2\rfloor)\\
& \iff \theta_1<\theta_2\textup{ or }(\theta_1=\theta_2\textup{ and }\lfloor 1-a\alpha_1\rfloor\leq\lfloor 1-a\alpha_2\rfloor)\\
& \iff\theta_1-\lfloor 1-a\alpha_1\rfloor\preceq\theta_2-\lfloor 1-a\alpha_2\rfloor\\
&\iff \Gamma_1\preceq\Gamma_2\,.
\end{align*}
Indeed, we have $\theta_i\in (0,1]$, which implies that $\langle\theta_i-\lfloor 1-a\alpha_i\rfloor\rangle=\theta_i$, while \eqref{eq: thetai} 
implies that $\theta_i-\lfloor 1-a\alpha_i\rfloor=\Gamma_i$. 
This proves the first part of the proposition. 

\medskip

Now, assume that $\Gamma_1=\Gamma_2$ so that $\theta_1-\lfloor 1-a\alpha_1\rfloor=\theta_2-\lfloor 1-a\alpha_2\rfloor$. Since $\theta_i\in (0,1]$, it follows that 
$\theta_1=\theta_2$. Hence $\langle\alpha_1\rangle=\langle\alpha_2\rangle$ by \eqref{eq:FromThetatoAlpha}. We obtain that 
$\lfloor 1-a\alpha_1\rfloor=\lfloor 1-a\alpha_2\rfloor$ and \eqref{eq: aalpha} implies that $\lfloor 1-\alpha_1\rfloor=\lfloor 1-\alpha_2\rfloor$. Since $\alpha_i=\langle\alpha_i\rangle-\lfloor 1-\alpha_i\rfloor$, we get $\alpha_1=\alpha_2$, as expected. This ends the proof. 
\end{proof}

\subsection{Efficient criteria for $q$-integrality and proof of Theorem~\ref{thm:N-integralityPolyn} }\label{sec:Efficient}

We are now ready to prove Theorem~\ref{thm:N-integralityPolyn}. 
The last missing ingredient is the following lemma.

\begin{lem}\label{lem: deltaxi} 
Let $\mathbf{r}=((r_1,s_1),\ldots,(r_v,s_v))$ and $\mathbf{t}=((t_1,u_1),\ldots,(t_w,u_w))$ be two vectors with integer coordinates such that, for all $(i,j)$, $s_iu_j\neq 0$ and the ratios $r_i/s_i$ and $t_j/u_j$ do not belong to $\mathbb{Z}_{\leq 0}$. Then the two following assertions are equivalent. 
\begin{itemize}
\item[$\mathrm{(i)}$] For all but finitely many $b$, $\Delta_b^{\mathbf{r},\mathbf{t}}$ is non-negative on $\mathbb R_{\geq 0}$. 
\item[$\mathrm{(ii)}$] For every $b\in\{1,\dots,d_{\mathbf{r},\mathbf{t}}\}$ and all $x\in\mathbb{R}$, we have $\Xi_{\mathbf{r},\mathbf{t}}(b,x)\geq 0$.
\end{itemize}
\end{lem}

\begin{proof}
We write $\Delta_b$ and $\Xi(b,\cdot)$ as respective short-hands for $\Delta_b^{\mathbf{r},\mathbf{t}}$ and $\Xi_{\mathbf{r},\mathbf{t}}(b,\cdot)$.

If $b$ is large enough, then we infer from \eqref{eq: deltanew} that $\Delta_b$ is a step function whose jumps on $[0,1]$ are precisely located 
at rationals of the form
$$
\gamma(r,s,k):=\frac{D_{b/c}(\alpha)+k}{c}+\frac{\lfloor 1-\alpha\rfloor}{b}\,,
$$
where $(r,s)$ belongs either to $\mathbf{r}$ or to $\mathbf{t}$, $\alpha=r/s$, $c=\gcd(r,s,b)$ and $k\in\{0,\dots,c-1\}$. 
More precisely,  $\Delta_b$ has a jump of positive amplitude at each element of the multiset  
$$
\mathfrak J_{b}^+:=\left\{\!\!\left\{\frac{D_{b/c_i}(\alpha_i)+k}{c_i}+\frac{\lfloor 1-\alpha_i\rfloor}{b}\,:\,i\in V_b,\,0\leq k\leq c_i-1\right\}\!\!\right\} \,.
$$
The amplitude of such a jump is equal to the multiplicity of the corresponding element in $\mathfrak J_{b}^+$. Similarly, 
$\Delta_b$ has a jump of negative amplitude at each element of the multiset  
$$
\mathfrak J_b^-:= \left\{\!\!\left\{\frac{D_{b/d_j}(\beta_j)+\ell}{d_j}+\frac{\lfloor 1-\beta_j\rfloor}{b}\,:\,j\in W_b,\,0\leq \ell\leq d_j-1\right\}\!\!\right\}\,,
$$
and the amplitude of such a jump is equal to the multiplicity of the corresponding element in $\mathfrak J_{b}^-$. 
By Lemma \ref{lem: gamma}, the supports of these multisets are included in $(0,1]$. 
Let  
$$
0<\gamma_1<\cdots<\gamma_\mu\leq 1
$$ 
denote the elements of the support of the multiset 
$\mathfrak J_b:=\mathfrak J_b^+\cup \mathfrak J_b^-$. We let $m_i^+$ (resp.\ $m_i^-$) denote the multiplicity of 
$\gamma_i$ in $\mathfrak J_{b}^+$ (resp.\ in $\mathfrak J_{b}^-$), and we set $m_i:=m_i^+-m_i^-$.  
Let $x\in [0,1]$ and set $\nu:=\sup\{i\in\{1,\ldots,\mu\} : \gamma_i\leq x\}$ with the convention $\sup(\emptyset)=-\infty$. Then, while setting $\gamma_{-\infty}:=0$, we obtain that 
$$\Delta_b(x)=\Delta_b(\gamma_\nu)=
\left\{\begin{array}{cl}
m_1+\cdots+m_{\nu} & \textup{if $\nu\geq 1$\,,}\\
0 & \textup{if $\nu=-\infty$\,.}
\end{array}\right.
$$

On the other hand, let $\underline b$ denote the  unique representative of $b$ in $\{1,\ldots,d_{\mathbf{r},\mathbf{t}}\}$ modulo $d_{\mathbf{r},\mathbf{t}}$ and let us consider the multisets 
$$
\mathcal J_{\underline b}^+:= \left\{\!\!\left\{\frac{\langle e_i\alpha_i\rangle+k}{c_i}-\lfloor 1-a\alpha_i\rfloor\,:\,i\in V_{\underline b},\,0\leq k\leq c_i-1\right\}\!\!\right\}
$$
and
$$
\mathcal J_{\underline b}^-:= \left\{\!\!\left\{\frac{\langle f_j\beta_j\rangle+\ell}{d_j}-\lfloor 1-a\beta_j\rfloor\,:\,j\in W_{\underline b},\,0\leq \ell\leq d_j-1\right\}\!\!\right\}\,.
$$
By Lemma \ref{lem: sauts}, 
the support of  $\mathcal J_{\underline b}:=\mathcal J_{\underline b}^+\cup \mathcal J_{\underline b}^-$ has also  cardinality $\mu$. 
Let 
$$
\Gamma_1\prec \cdots\prec \Gamma_\mu
$$ 
denote the elements of the support of  
$\mathcal J_{\underline b}$ (ordered with respect to Christol order).  
Furthermore, Lemma \ref{lem: sauts} implies that $\Gamma_i$ has also multiplicity $m_i^+$ in $\mathcal J_{\underline b}^+$ and $m_i^-$ in $\mathcal J_{\underline b}^-$.  
Let $x\in \mathbb R$ and set $\nu:=\sup\{i\in\{1,\ldots,\mu\} : \Gamma_i\preceq x\}$. Then, while setting $\Gamma_{-\infty}:=0$, we obtain that 
$$\Xi(\underline{b},x)=\Xi(\underline{b},\Gamma_\nu)=
\left\{\begin{array}{cl}
m_1+\cdots+m_\nu & \textup{if $\nu\geq 1$\,,}\\
0 & \textup{if $\nu=-\infty$\,.}
\end{array}\right.
$$

We deduce that 
 \begin{equation}\label{eq: deltax}
 \Delta_b([0,1])=\{0,\Delta_b(\gamma_1),\ldots,\Delta_b(\gamma_\mu)\}=\{0,\Xi(\underline{b},\Gamma_1),\ldots,\Xi(\underline{b},\Gamma_\mu)\}=\Xi(\underline{b},\mathbb R)\,.
 \end{equation}
This shows that Assertion~$\mathrm{(ii)}$ is equivalent to the fact that  $\Delta_b$ is non-negative on $[0,1]$ for all $b$ large enough. 
 On the other hand, the identity $\Delta_b(x+k)=\Delta_b(x)+k\Delta_b(1)$ proved for $k\in\mathbb Z$ in Lemma \ref{lem:DeltaJumps} shows that 
 $\Delta_b$ is non-negative on $[0,1]$ if and only if it is non-negative on $\mathbb R_{\geq 0}$. 
 In the end, we obtain that Assertion~$\mathrm{(i)}$ and Assertion~$\mathrm{(ii)}$ are equivalent, which ends the proof.  
\end{proof} 

We first deduce form Proposition \ref{propo: N-integrality} and Lemma \ref{lem: deltaxi} the following result. 

\begin{thm}\label{thm:N-integrality}
Let us assume that $(Q_{\mathbf{r},\mathbf{t}}(q;n))_{n\geq 0}$ is a well-defined sequence which is not eventually zero. Then the two following assertions are equivalent.
\begin{itemize}
\item[$\mathrm{(i)}$] There exists $C(q)\in\mathbb{Z}[q]\setminus\{0\}$ such that, for every $n\geq 0$, $C(q)^nQ_{\mathbf{r},\mathbf{t}}(q;n)\in\mathbb{Z}[q^{-1},q]$.

\item[$\mathrm{(ii)}$] For every $b\in\{1,\dots,d_{\mathbf{r},\mathbf{t}}\}$ and all $x\in\mathbb{R}$, we have $\Xi_{\mathbf{r},\mathbf{t}}(b,x)\geq 0$.
\end{itemize}
\end{thm}

Finally, we can achieve the proof of our main $q$-integrality criterion.  

\begin{proof}[Proof of Theorem~\ref{thm:N-integralityPolyn}]
The result is a straightforward consequence of Corollary \ref{coro:integrality} and Lemma \ref{lem: deltaxi}. 
\end{proof}

As discussed in Section \ref{sec: first criteria},  efficient criteria for the $q$-integrality of the sequences $(Q_{\mathbf{r},\mathbf{t}}(q;n))_{n\leq 0}$ and $(Q_{\mathbf{r},\mathbf{t}}(q;n))_{n\in\mathbb{Z}}$ can also be derived from Theorems \ref{thm:N-integralityPolyn} and \ref{thm:N-integrality}.

\section{Examples and applications}\label{sec: ex}

In this last section, we give an overview of the computation of Christol step functions through some classical examples.

\subsection{General considerations}\label{sec:general}

We continue with the general notation of this paper. 
For every $b\in\{1,\dots,d_{\mathbf{r},\mathbf{t}}\}$, we have defined in Section \ref{sec:fcChristol} the step function $\Xi_{\mathbf{r},\mathbf{t}}(b,\cdot)$. 
Using the notation used in the proof of Lemma \ref{lem: deltaxi}, we obtain that 
$$
\Xi_{\mathbf{r},\mathbf{t}}(b,x):=\#\{\!\{\gamma\in \mathcal J_b^+\,:\,\gamma\preceq x\}\!\}-\#\{\!\{\gamma\in \mathcal J_b^-\,:\,\gamma\preceq x\}\!\}\,.
$$
When $b\in\{1,\dots,d_{\mathbf{r},\mathbf{t}}\}$ is coprime to $d_{\mathbf{r},\mathbf{t}}$, the function $\Xi_{\mathbf{r},\mathbf{t}}(b,\cdot)$ 
is  easier to compute since we have 
$$
\mathcal J_b^+=\{\!\{ a\alpha_1,\dots,a\alpha_v\}\!\}\quad\textup{and}\quad \mathcal J_b^-=\{\!\{a\beta_1,\dots,a\beta_w\}\!\}\,,
$$
where $a$ is the unique integer in $\{1,\dots,d_{\mathbf{r},\mathbf{t}}\}$ satisfying $ab\equiv 1\mod d_{\mathbf{r},\mathbf{t}}$. 
Theorem C can then be rephrased as follows: 
$$(Q_{\boldsymbol{\alpha},\boldsymbol{\beta}}(n))_{n\geq 0} \mbox{ is $N$-integral } \iff \forall  b\in \{1,\ldots, d_{\mathbf{r},\mathbf{t}}\} \mid  
\gcd(b,d_{\mathbf{r},\mathbf{t}})=1:  \Xi_{\mathbf{r},\mathbf{t}}(b,x)\geq 0\,, \, 
\forall x\in\mathbb{R}\,.
$$
 
Starting with an $N$-integral hypergeometric sequence 
$$
\frac{(\alpha_1)_n\cdots (\alpha_v)_n}{(\beta_1)_n\cdots (\beta_w)_n} \quad\quad n\geq 0
$$
and taking $\mathbf{r}$ and $\mathbf{t}$ such that
$$
Q_{\mathbf{r},\mathbf{t}}(q;n)=\frac{(q^{r_1};q^{s_1})_n\cdots(q^{r_v};q^{s_v})_n}{(q^{t_1};q^{u_1})_n\cdots(q^{t_w};q^{u_w})_n}\,,
$$
with $r_i/s_i=\alpha_i$ and $t_j/u_j=\beta_j$, Lemma \ref{lem: sauts} ensures the existence of a constant $\mathfrak{c}_{\mathbf{r},\mathbf{t}}$ such that, 
for every integer $b$ coprime to $d_{\mathbf{r},\mathbf{t}}$ and larger than $\mathfrak{c}_{\mathbf{r},\mathbf{t}}$, we have 
$$
v_{\phi_b}(Q_{\mathbf{r},\mathbf{t}}(q;n))=\Delta^{\mathbf{r},\mathbf{t}}_b(n/b)\geq 0\,.
$$ 
Indeed, for $b>\mathfrak{c}_{\mathbf{r},\mathbf{t}}$, Lemma \ref{lem: sauts} shows that the $\leq$-ordering of the jumps of 
$\Delta_b^{\mathbf{r},\mathbf{t}}$ on $[0,1]$ is the same as the $\preceq$-ordering of the ones of $\Xi_{\mathbf{r},\mathbf{t}}(\underline{b},\cdot)$ on $\mathbb{R}$, 
where $\underline{b}$ is the unique representative in $\{1,\dots,d_{\mathbf{r},\mathbf{t}}\}$ of $b$ modulo $d_{\mathbf{r},\mathbf{t}}$. 
In particular, $\Delta_b^{\mathbf{r},\mathbf{t}}$ is non-negative on $\mathbb{R}_{\geq 0}$ as expected.

Hence the denominator of $Q_{\mathbf{r},\mathbf{t}}(q;n)$ could only contain cyclotomic polynomials $\phi_b(q)$ with $b\leq\mathfrak{c}_{\mathbf{r},\mathbf{t}}$ 
or $b$ not coprime to $d_{\mathbf{r},\mathbf{t}}$. The situation with such numbers $b$ is much more complicated and strongly depends on the
 $\gcd$'s of the pairs $(r_i,s_i)$ and $(t_j,u_j)$. 

\medskip

Let us first consider the case where $\gcd(r_i,s_i)=1$ and $\gcd(t_j,u_j)=1$ for all $i$ and $j$. Let $b\in\{1,\dots,d_{\mathbf{r},\mathbf{t}}\}$, 
$\tilde{b}$ be the greatest divisor of $b$ coprime to $d_{\mathbf{r},\mathbf{t}}$, and let $a$ be the unique integer in $\{1,\dots,d_{\mathbf{r},\mathbf{t}}\}$ 
satisfying $a\tilde{b}\equiv 1\mod d_{\mathbf{r},\mathbf{t}}$. Then, following the notation of Section \ref{sec:fcChristol}, we find  $c_i=d_j=1$, so that 
$$
V_b:=\{1\leq i\leq v\,:\,\gcd(s_i,b)=1\}\quad\textup{and}\quad W_b:=\{1\leq j\leq w\,:\,\gcd(u_j,b)=1\}\,,
$$
which yields 
$$
\mathcal J^+_b:=\{\!\left\{\langle e_i\alpha_i\rangle-\lfloor 1-a\alpha_i\rfloor\,:\,i\in V_b\right\}\!\}\quad\textup{and}\quad \mathcal J_b^-:=\{\!\left\{\langle f_j\beta_j\rangle-\lfloor 1-a\beta_j\rfloor\,:\,j\in W_b\right\}\!\}\,.
$$
Hence each \og classical \fg{} jump occurring at $a\alpha_i$ (by this, we mean the jumps occurring when $b$ is coprime to $d_{\mathbf{r},\mathbf{t}}$) either disappears because $b$ is not coprime to $s_i$, or is  replaced by a jump at $\langle e_i\alpha_i\rangle-\lfloor 1-a\alpha_i\rfloor$ when $b$ is coprime to $s_i$. Even in this particular case, 
we already understand that the 
new step functions can behave in a very different way than the classical ones. 

As an illustration, we consider the simple example 
$$
Q_{\mathbf{r},\mathbf{t}}(q;n):=\frac{(q;q^3)_n(q^2;q^3)_n}{(q;q^2)_n(q;q)_n}\,,
$$
which was introduced at the end of Section \ref{sec:mainresults}  and corresponds to $\mathbf{r}=((1,3),(2,3))$ and $\mathbf{t}=((1,2),(1,1))$. 
We have
$$
\left(\frac{(1-q^2)(1-q)}{(1-q^3)^{2}}\right)^nQ_{\mathbf{r},\mathbf{t}}(q;n)\underset{q\rightarrow 1}{\longrightarrow}\frac{(1/3)_n(2/3)_n}{(1/2)_n(1)_n}\,,
$$
the right-hand side being $N$-integral. This can be derived from Formula \eqref{eq: dn!}. 
We find that $d_{\mathbf{r},\mathbf{t}}=6$, and for $b=3$ we obtain $\tilde{b}=1$ and $a=1$. This yields $V_3=\emptyset$, $W_3=\{1,2\}$, and $f_1=f_2=1$. 
Hence $\mathcal J_3^+=\emptyset$ and $\mathcal J_3^-=\{\!\{1/2,1\}\!\}$, so that $\Xi(3,1/2)<0$. Thus, we deduce from Theorem~\ref{thm:N-integralityPolyn} that  
the sequence $(Q_{\mathbf{r},\mathbf{t}}(q;n))_{n\geq 0}$ is not $q$-integral. 

On the other hand, we have 
\begin{equation}\label{eq:binom32}
\frac{(q;q^3)_n(q^2;q^3)_n}{(q;q^2)_n(q;q)_n}\cdot\frac{(q^3;q^3)_n}{(q^2;q^2)_n}=\qbi{3n}{2n}{q}\in\mathbb{Z}[q]\,,
\end{equation}
which shows that  the corresponding $q$-hypergeometric sequence is obviously $q$-integral. 
In order to understand the effect of the extra factors $(q^3;q^3)_n$ and $(q^2;q^2)_n$, 
we have to investigate the case where $\gcd(r_i,s_i)\neq 1$.

When $\gcd(r_i,s_i)\neq 1$, we possibly have $c_i=\gcd(r_i,s_i,b)\neq 1$. In this case, either $\gcd(s_i,b)\neq c_i$ and the \og classical \fg{} jump at 
$a\alpha_i$ disappears, or there is an integer $e_i$ satisfying $be_i\equiv c_i\mod s_i$ and the jump at $a\alpha_i$ splits into  $c_i$ distinct jumps at 
$$
\frac{\langle e_i\alpha_i\rangle+k}{c_i}-\lfloor 1-a\alpha_i\rfloor\,,\quad\quad 0\leq k\leq c_i-1\,.
$$ 

Let us now return to \eqref{eq:binom32} and consider the case where $b=3$ . Then, we find that $c_3=3$,  $V_3=\{3\}$, and  $e_3=1$. 
This yields jumps with amplitude $+1$ at all elements of the (multi)set $\mathcal J_3^+=\{\!\{1/3, 2/3, 1\}\!\}$. 
On the other hand, we have $W_3=\{1,2,3\}$ and $f_1=f_2=f_3=1$, which yields 
jumps with amplitude $-1$ at all elements of the multiset $\mathcal J_3^-=\{\!\{1/2,1,1\}\!\}$. In the end, we get that 
\begin{equation}\label{eq: prec3}
\Gamma_1=\frac{1}{3}\prec \Gamma_2=\frac{1}{2}\prec\Gamma_3 =\frac{2}{3}\prec \Gamma_4=1\,,
\end{equation}
with $m_1=1$, $m_2=-1$, $m_3=1$, and $m_4=-1$. It follows that the step function $\Xi(3,\cdot)$ is non-negative on $\mathbb{R}$, as expected.

\subsection{$q$-Factorial ratios}

Let us recall that 
$[n]_q=(1-q^n)/(1-q)$, so that 
$$
[n]_q=\prod_{b\geq 2,\,b\mid n}\phi_b(q) 
$$
and 
\begin{equation}\label{eq: n!}
[n]!_q:=\prod_{i=1}^n\frac{1-q^i}{1-q}=\prod_{b\geq 2,\,b\mid n}\phi_b(q)^{\lfloor n/b\rfloor}\,.
\end{equation}
Given two vectors $e:=(e_1,\ldots,e_v)$ and $f:=(f_1,\ldots,f_w)$ whose coordinates are positive integers, 
we define as in~\cite{WZ} the $q$-analog of the factorial ratio $Q_{e,f}(n)$ as 
\begin{equation*}
Q_{e,f}(q; n):=\frac{[e_1n]!_q\cdots[e_vn]!_q}{[f_1n]!_q\cdots[f_wn]!_q}\,\cdot
\end{equation*}
We deduce from \eqref{eq: n!} that 
\begin{equation*}
 Q_{e,f}(q; n)=\prod_{b\geq 2,\,b\mid n}\phi_b(q)^{\Delta_{e,f}(n/b)}\,,
\end{equation*}
where 
$$
\Delta_{e,f}(x)=\sum_{i=1}^v\lfloor e_ix\rfloor-\sum_{j=1}^w\lfloor f_j x\rfloor 
$$
 is the classical Landau function, as defined  in~\eqref{eq:deltadef}. 
We easily obtain that $Q_{e,f}(q;n)$ is $q$-integral if and only if $\Delta_{e,f}$ is non-negative on $[0,1]$. 
Note that these properties are also equivalent to the fact that $Q_{e,f}(q;n)\in\mathbb Z[q]$ (see also~\cite{WZ} where a positivity conjecture of the coefficients of these polynomials is proposed). 
It is therefore much more efficient  to work with $\Delta_{e,f}$  than to compute the corresponding Christol functions.  

The example given in  \eqref{eq:binom32} corresponds to  $e=(3)$ and $f=(2,1)$, so that  
$$
\Delta_{e,f}(x)=\lfloor 3x\rfloor-\lfloor 2x\rfloor-\lfloor x\rfloor\,.
$$
On $[0,1]$, this step function has jumps with positive amplitude $+1$ at $1/3$ and $2/3$, and jumps with negative amplitude $-1$ at $1/2$ and $1$. 
As expected, 
we retrieve the same ordering as in \eqref{eq: prec3} for the jumps of $\Xi(3,\cdot)$.

\subsection{A famous non-factorial example}\label{sec:ChristolExample}

When introducing his step functions in  \cite{Christol}, 
Christol was motivated by the following question:  is it true that an $N$-integral hypergeometric series is the diagonal of a rational fraction in several variables?  
The hypergeometric sequence 
\begin{equation}\label{eq:ChristolExample}
\frac{(1/9)_n(4/9)_n(5/9)_n}{(1/3)_n(1)_n^2} \quad\quad n\geq 0 
\end{equation}
is one of the simplest examples of an $N$-integral hypergeometric sequence for which the question is still open 
(although recent progress in this direction has been made in \cite{AKM20} and \cite{BS22}).  

In this case, the six Christol functions associated with each $b$ coprime to $9$ are non-negative on $\mathbb{R}$.  
By Theorem C, this ensures that this hypergeometric sequence is $N$-integral. A precise formula for the smallest positive integer $N_0$ 
is given in \cite[Theorem~$4$]{DRR}: here we get that $N_0=9^3$.

As already discussed, a natural $q$-analog of \eqref{eq:ChristolExample} can be defined as 
$$
\left(\frac{(1-q^3)(1-q)^2}{(1-q^9)^3}\right)^n\frac{(q;q^9)_n(q^4;q^9)_n(q^5;q^9)_n}{(q;q^3)_n(q;q)_n^2} \quad\quad n\geq 0 \,.
$$
The $q$-integrality of this sequence is equivalent to the one of the $q$-hypergeometric sequence 
$Q_{\mathbf{r},\mathbf{t}}(q;n)$, where $\mathbf{r}=((1,9),(4,9),(5,9))$ and $\mathbf{t}=((1,3),(1,1),(1,1))$. 

It remains to consider the Christol functions associated with $b\in\{3,6,9\}$. 
For $b=3$, we have $\gcd(9,b)=3\neq 1$ so that $\mathcal J_3^+=\emptyset$. But due to the factors $(q;q)_n^2$ in the denominator, we obtain that 
$\mathcal J_3^-=\{\!\{1,1\}\!\}$, so that $\Xi_{\mathbf{r},\mathbf{t}}(b,1)<0$.  
We deduce from Theorem \ref{thm:N-integralityPolyn} that the sequence $(Q_{\mathbf{r},\mathbf{t}}(q;n))_{n\geq 0}$ is not $q$-integral. 
In this example, all the ``classical jumps" with positive amplitude have disappeared for $b=3$. 

In fact, we can retrieve $q$-integrality by  adding a factor $(q^9;q^9)_n$ to the numerator and a factor $(q;q)_n$ to the denominator. This  
 leads to the slightly modified $q$-analog: 
$$
\left(\frac{(1-q^3)(1-q)^3}{(1-q^9)^4}\right)^n\frac{(q;q^9)_n(q^4;q^9)_n(q^5;q^9)_n(q^9;q^9)_n}{(q;q^3)_n(q;q)_n^3} \quad\quad n\geq 0\,.
$$
With this new choice of parameters $\mathbf{r'}$ and $\mathbf{t'}$, the functions $\Xi_{\mathbf{r'},\mathbf{t'}}(b,\cdot)$ for $b$ coprime to $9$ remains unchanged. 
However, for $b$ in $\{3,6,9\}$, one finds that $V_b$ is no longer empty. A computation shows that $V_b=\{4\}$, $W_b=\{2,3,4\}$, 
$\mathcal J_3^-=\mathcal J_9^-=\{\!\{1,1,1\}\!\}$, $\mathcal J_6^-=\{\!\{5,5,5\}\!\}$, while 
$$
\mathcal J_3^+=\left\{\!\!\left\{\frac{1}{3},\frac{2}{3},1\right\}\!\!\right\},\quad \mathcal J_6^+=\left\{\!\!\left\{\frac{1}{3}+4,\frac{2}{3}+4,5\right\}\!\!\right\}\,,\quad\textup{and}\quad 
\mathcal J_9^+=\left\{\!\!\left\{\frac{1}{9},\frac{2}{9},\dots,\frac{8}{9},1\right\}\!\!\right\}\, .
$$ 
In all cases, $\Xi_{\mathbf{r'},\mathbf{t'}}(b,\cdot)$ is now non-negative on $\mathbb{R}$ and we infer from Theorem \ref{thm:N-integrality}  that the sequence 
$(Q_{\mathbf{r'},\mathbf{t'}}(q;n))_{n\geq 0}$ is $q$-integral.

Finally, we consider a third $q$-analog of the hypergeometric sequence \eqref{eq:ChristolExample}, which we define as 
\begin{equation}\label{eq:reftilde}
\widetilde{Q}_{\boldsymbol{\alpha},\boldsymbol{\beta}}(q;n)=\frac{(q^{1/9};q)_n(q^{4/9};q)_n(q^{5/9};q)_n}{(q^{1/3};q)_n(q;q)_n^2}\, \quad\quad n\geq 0\,.
\end{equation}
As already discussed, the $q^{1/9}$-integrality of $(\widetilde{Q}_{\boldsymbol{\alpha},\boldsymbol{\beta}}(q;n))_{n\geq 0}$ is equivalent to the $q$-integrality of the sequence 
$$
\widetilde{Q}_{\boldsymbol{\alpha},\boldsymbol{\beta}}(q^9;n)=\frac{(q;q^9)_n(q^4;q^9)_n(q^5;q^9)_n}{(q^3;q^9)_n(q^9;q^9)_n^2} \quad\quad n\geq 0\,.
$$
Furthermore, we have $\widetilde{Q}_{\boldsymbol{\alpha},\boldsymbol{\beta}}(q^9;n)=Q_{\mathbf{r},\mathbf{t}}(q;n)$ for a suitable choice of 
vectors $\mathbf{r}$ and $\mathbf{t}$. As previously, 
a computation shows that for $b=3$, we have $\mathcal J^+_3=\emptyset$ while $1\in \mathcal J_3^-$, so that 
$\Xi_{\mathbf{r},\mathbf{t}}(3,1)<0$. We deduce from Theorem \ref{thm:N-integrality} that $(Q_{\mathbf{r},\mathbf{t}}(q;n))_{n\geq 0}$ is not $q$-integral. Then, it  
follows that the sequence defined in \eqref{eq:reftilde} is not $q^{1/9}$-integral.

We observe that, in this case, we cannot use the same trick as before. Indeed, multiplying $\widetilde{Q}_{\mathbf{r},\mathbf{t}}(q^9;n)$ by $(q^9;q^9)_n/(q;q)_n$ amounts to multiplying \eqref{eq:reftilde} by $(q;q)_n/(q^{1/9};q^{1/9})_n$ which does not correspond to any choice of parameters $\boldsymbol{\alpha}$ and $\boldsymbol{\beta}$.

\end{document}